\documentclass[10pt,reqno]{amsart}

\usepackage{amssymb}
\usepackage{amsfonts}
\usepackage{amsthm}
\usepackage{amsmath}
\usepackage{bbm}
\usepackage{xcolor}
\usepackage{tikz}
\usepackage{hyperref}
\usepackage{ulem}

\numberwithin{equation}{section}
\allowdisplaybreaks

\newtheorem{theorem}{Theorem}[section]
\newtheorem{proposition}[theorem]{Proposition}

\newtheorem{lemma}[theorem]{Lemma}
\newtheorem{corollary}[theorem]{Corollary}

\theoremstyle{definition}

\theoremstyle{remark}

\newcommand{\R}{\mathbb{R}}

\newcommand{\N}{\mathbb{N}}

\renewcommand{\S}{\mathbb{S}}
\renewcommand{\P}{\mathbb{P}}

\renewcommand{\hat}{\widehat}
\newcommand{\eps}{\varepsilon}

\newcommand{\scriptD}{\mathcal{D}}
\newcommand{\scriptE}{\mathcal{E}}

\newcommand{\scriptJ}{\mathcal{J}}

\newcommand{\scriptR}{\mathcal{R}}
\newcommand{\scriptS}{\mathcal{S}}

\newcommand{\jp}[1]{\langle{#1}\rangle}
\newcommand{\qtq}[1]{\:\text{#1}\:}

\DeclareMathOperator*{\ch}{ch}
\DeclareMathOperator*{\wklim}{wk-lim}

\DeclareMathOperator*{\supp}{supp}
\DeclareMathOperator*{\diam}{diam}
\DeclareMathOperator*{\dist}{dist}

 \usepackage[shortlabels]{enumitem}
                    \setlist[enumerate, 1]{1\textsuperscript{o}}

\begin{document}
\title[Extremizing adjoint Fourier restriction to the sphere]{On extremizing sequences for adjoint Fourier restriction to the sphere}
\author{Taryn C Flock and Betsy Stovall}
\begin{abstract}
In this article, we develop a linear profile decomposition for the $L^p \to L^q$ adjoint Fourier restriction operator associated to the sphere, valid for exponent pairs $p<q$ for which this operator is bounded.  Such theorems are new when $p \neq 2$.  We apply these methods to prove new results regarding the existence of extremizers and the behavior of extremizing sequences for the spherical extension operator.  Namely, assuming boundedness, extremizers exist if  $q>\max\{p,\tfrac{d+2}d p'\}$, or if $q=\tfrac{d+2}d p'$ and the operator norm exceeds a certain constant times the operator norm of the parabolic extension operator.  
\end{abstract}

\maketitle

\section{Introduction and statement of results}

We  consider the Fourier restriction and extension operators associated to the unit sphere $\S^d \subseteq \R^{d+1}$, which are given by
$$
\scriptR g(\omega) := \int_{\R^d} e^{-ix\omega}g(x)\, dx, \: \omega \in \S^d, \qquad \scriptE f(x) = \int_{\S^d} e^{ix\omega}f(\omega)\, d\sigma(\omega),
$$
for $g$ in the Schwartz class, $\scriptS(\R^{d+1})$, and $f \in C^\infty(\S^d)$.  Here $\sigma$ denotes $d$-dimensional Hausdorff measure on $\S^d$.  Despite decades of study, the precise conditions on exponents $p$ and $q$ for which (say) $\scriptE$ extends as a bounded linear operator from $L^p(\S^d)$ to $L^q(\R^{1+d})$ are not fully resolved for any $d \geq 2$.  

We do not seek to directly address such questions.  Rather, we ask, under the assumption of $L^p \to L^q$ boundedness of $\scriptE$, what is the behavior of bounded sequences $\{f_n\}$ whose extensions $\{\scriptE f_n\}$ do not converge to zero in norm.  This will lead us to develop a qualitative description of such sequences called a profile decomposition.  A particular scenario of interest is when the sequence $\{f_n\}$ is both $L^p$-normalized ($\|f_n\|_{L^p(\S^d;d\sigma)} \equiv 1$) and extremizing ($\|\scriptE f_n\|_{L^q(\R^{1+d})} \to \|\scriptE\|_{L^p(\S^d;d\sigma) \to L^q(\R^{1+d})}$), in which case our profile decompositions provide quite a bit of information (at least when $q>p$).  

In order to state our results, we will need some notation and terminology.  Noting that $\scriptR$ and $\scriptE$ are dual to one another, we denote their (common) operator norm by
$$
S_{p \to q} := \sup_{\|f\|_{L^p(\S^d;d\sigma)} = 1} \|\scriptE f\|_{L^q(\R^{d+1})} = \sup_{\|g\|_{L^{q'}(\R^{d+1})} = 1} \|\scriptR g\|_{L^{p'}(\S^d;d\sigma)},
$$
where the suprema are taken over (e.g.) smooth, compactly supported functions.   These operator norms are conjectured to be finite whenever $q \geq \tfrac{d+2}d p'$ and $q > \tfrac{2(d+1)}d$ both hold, and these conditions are known to be necessary.  

We are interested in the questions of whether there exist nonzero functions for which equality holds in the restriction/extension inequalities
\begin{equation} \label{E:restriction/extension}
\|\scriptR g\|_{L^{p'}(\S^d;d\sigma)} \leq S_{p \to q} \|g\|_{L^{q'}(\R^{d+1})}, \qquad \|\scriptE f\|_{L^q(\R^{d+1})} \leq S_{p \to q} \|f\|_{L^p(\S^d;d\sigma)},
\end{equation}
in cases where the operator norms are finite, and whether $L^p$-normalized extremizing sequences are convergent in some sense.  We are further interested in connections between these operators and the restriction/extension operators associated to the paraboloid, $\P:=\{(\tfrac12|\xi|^2,\xi):\xi \in \R^d\}$:  
$$
\scriptR_\P g(\xi):= \int_{\R^{d+1}} e^{-i(\frac12|\xi|^2,\xi) x}g(x)\, dx, \: \xi \in \R^d, \qquad \scriptE_\P f(x):= \int_{\R^d} e^{ix(\frac12|\xi|^2,\xi)}f(\xi)\, d\xi,
$$
whose common operator norms we denote by
\begin{equation} \label{E:rest/extn parab}
P_{p \to q} :=  \sup_{\|f\|_{L^p(\R^d)} = 1} \|\scriptE_\P f\|_{L^q(\R^{d+1})} = \sup_{\|g\|_{L^{q'}(\R^{d+1})} = 1} \|\scriptR_\P g\|_{L^{p'}(\R^d)}.
\end{equation}

More generally, we extend the profile decomposition methods of Fanelli--Visciglia--Vega \cite{FVV2011} for the case $p=2$, $q> \tfrac{d+2}dp'$ and Frank--Lieb--Sabin in \cite{FLS} for the case $p=2$, $q=\tfrac{d+2}dp'$, to the region $p < q \leq \tfrac{d+2}d p'$, $p\neq 1$. (A more complete history is given in the next section.) 

Before stating our results, we set some basic terminology. Fix a pair of exponents $(p,q)$.  A sequence $(f_n)$ in $L^p(\S^d)$ is \textit{$L^p$ normalized} if $\|f_n\|_p = 1$ for all $n$ and is \textit{extremizing} if $\lim_{n \to \infty} \|\scriptE f_n\|_q/\|f_n\|_p = S_{p \to q}$.  (We note that normalized extremizing sequences exist for every exponent pair $(p,q)$; whether they converge and what are their properties are more subtle questions.)

Many of our results are partly conditional on progress toward the (adjoint) restriction conjecture for the sphere.  We adopt the following convention, which will make for cleaner statements later on:  We say that the extension conjecture holds at $(p,q) \in [1,\infty]^2$ if $S_{p \to q}<\infty$, if $q \leq \tfrac{2(d+1)}d$, or if $q < \tfrac{d+2}d p'$.  (Of course, in the latter two cases, the extension operator is known to be unbounded.)  In the non-vacuous range, the conjecture has been verified for all $(p,q)$ when $d=1$ \cite{Fefferman, Zygmund}, and, in higher dimensions, has been  verified on a neighborhood of the region $q \geq \tfrac{2(d+3)}{d+1}$, $q \geq \tfrac{d+2}d p'$ (see \cite{BassamShayya, Guth_d=2, Guth_dgeq3, HickmanRogers, TaoParab, TVV, HWang} for more precision regarding the current status).

Our results are cleanest off of the parabolic scaling line, wherein  H\"older's inequality rules out the possibility that extremizing sequences might concentrate.  

\begin{theorem} \label{T:off scaling}
Assume that $q>\max\{p,\tfrac{d+2}dp'\}$ and that the extension conjecture holds on a neighborhood of $(p,q)$.  Then every $L^p$-normalized extremizing sequence for the inequality $\|\scriptE f\|_q \leq S_{p \to q} \|f\|_p$ is precompact in $L^p$ after the application of an appropriate sequence of spacetime modulations.  In particular, extremizers exist for this inequality.  
\end{theorem}

In the case $p=2$, this result is due to Fanelli--Visciglia--Vega \cite{FVV2011}.  The hypothesis $q > p$ is likely an artifact of our proof, which uses a convexity argument.  Indeed, in certain special cases, such as when $p=\infty$ and $q$ is an even integer, one can use other means to prove the existence of extremizers \cite{CS}.

For inequalities with $(p^{-1},q^{-1})$ on the parabolic scaling line $\{q=\tfrac{d+2}d p'\}$, we cannot (yet) rule out the possibility of concentration.

\begin{theorem}\label{T:scaling}
Let $1 < p<\tfrac{2(d+1)}d$ and $q=\tfrac{d+2}d p'$, and assume that the extension conjecture for the sphere holds on a neighborhood of $(p,q)$.  Let $(f_n)$ be an $L^p(\S^d)$ normalized extremizing sequence of \eqref{E:restriction/extension}.  After passing to a subsequence, either:\\
(i) There exists a sequence $\{x_n\} \subseteq \R^{d+1}$ such that $e^{i x_n \omega}f_n(\omega)$ converges in $L^p(S^d)$ to an extremizer $f$ of \eqref{E:restriction/extension}, \\
or\\
(ii) There exist sequences $\{x_n\} \subseteq \R^{1+d}$, orthogonal transformations  $\{R_n\} \subseteq O(d)$, positive numbers $\lambda_n \searrow 0$, and functions $\phi^\pm \in L^p(\R^d)$ such that 
\begin{equation}\label{E:parab limit}
\lim_{n \to \infty}
\bigl\|e^{ix_n\omega}f_n(R_n\omega)
-\lambda_n^{-d/p} \sum_\pm \phi^\pm(\tfrac{\omega'}{\lambda_n})\chi_{\{\pm\omega_1>\frac12\}}\|_p = 0.
\end{equation}  
\end{theorem}

Two remarks are in order.  First, the $p=2$ case was already resolved in \cite{FLS}.  Second, when $p=1$, existence of extremizers and noncompactness modulo symmetries of normalized extremizing sequences (or even, sequences of normalized extremizers) are both elementary to prove, as any nonnegative $L^1$ function is extremal.

We say that (along a subsequence) $(f_n)$ converges modulo the modulation symmetry in Conclusion (i) and that $(f_n)$ concentrates antipodally and converges modulo translations, dilations (a nonsymmetry), and rotations in Conclusion (ii).  

We can improve upon Theorem~\ref{T:scaling} by estimating the operator norm in the  case of concentration, generalizing the main results of \cite{FLS} and \cite{ChristShao} (therein carried out in the $p=2$ case).  This will require some further notation.  

For $1 \leq p < q=\tfrac{d+2}dp'$, we define
\begin{equation} \label{E:def alpha}
\alpha_{p,q}:=\max_{t \in [0,1]} \frac{\|1+t e^{i\theta}\|_{L^q([0,2\pi],d\theta/2\pi)}}{(1+t^p)^{1/p}}.
\end{equation}
The parameter $t$ will arise as the ratio between the norms  of the extensions of two antipodally concentrating profiles.  Considering such pairs will lead us to a lower bound for $S_{p\to q}$. 

\begin{proposition} \label{P:op norms}
Let $1 \leq p < \tfrac{2(d+1)}d$ and set $q:=\tfrac{d+2}d p'$.  Then
\begin{equation} \label{E:op norms}
S_{p \to q} \geq \alpha_{p,q} P_{p \to q}.
\end{equation}
\end{proposition}

The quantity 
$$
\beta_{p,q}:= 2^{\frac1{r'}}\left(\frac{\Gamma(\frac{q+1}2)}{\sqrt \pi \Gamma(\frac{q+2}2)}\right)^{\frac1q}, \qquad r:=\max\{p,2\},
$$
seems somewhat easier to understand than $\alpha_{p,q}$, and we note the following relationship between the two.  

\begin{proposition} \label{P:alpha beta}
For $p\geq 2$,  \( \alpha_{p \to q}=\beta_{p\to q}\); while for $p<2$, \(  \alpha_{p \to q}<\beta_{p\to q}\). 
\end{proposition} 

 The transition at $p=2$ in Proposition~\ref{P:alpha beta} is connected with a bifurcation of our results along the parabolic scaling line into the cases $1 < p < 2$ and $2 < p <\tfrac{2(d+1)}d$. We begin with the latter case, in which our results are stronger. 

\begin{theorem} \label{T:op norms big p}
Let $2\leq p<\tfrac{2(d+1)}d$ and $q=\tfrac{d+2}dp'$, and assume that the extension conjecture for the sphere holds on a neighborhood of $(p,q)$. 
If $ S_{p \to q} > \beta_{p,q} P_{p \to q}$,  then  extremizers exist for the extension operator in \eqref{E:restriction/extension} and all normalized extremizing sequences possess subsequences that converge in $L^p$, after modulation. 
Otherwise, $S_{p \to q} =  \beta_{p,q} P_{p \to q}$, and concentrating, extremizing sequences $(f_n)$ exist; after passing to a subsequence and normalizing, they must obey \eqref{E:parab limit}, for some $\phi^\pm$ extremal for $\scriptE_\P:L^p(\R^d) \to L^q(\R^{d_1})$ and obeying
 $$
|\scriptE \phi^+(x_1,x')| = |\scriptE \phi^-(-x_1,x')|. 
 $$
\end{theorem} 

When $p<2$ the gap between $\alpha_{p,q}$ and $\beta_{p,q}$ seen in Proposition~\ref{P:alpha beta} leaves some room for improvement in the following theorem, as discussed at the end of Section~\ref{S:equal profiles}.

\begin{theorem} \label{T:op norms small p}
Let $1<p<2$ and $q=\tfrac{d+2}dp'$, and assume that the extension conjecture for the sphere holds on a neighborhood of $(p,q)$. If $ S_{p \to q} \geq \beta_{p,q} P_{p \to q}$ then  extremizers exist for the extension operator in \eqref{E:restriction/extension} and all normalized extremizing sequences possess subsequences that converge in $L^p$, modulo spatial translations of the extension. If $ S_{p \to q} =  \alpha_{p,q} P_{p \to q}$, $\phi^\pm$ is extremal for \eqref{E:rest/extn parab}, and 
$$
|\scriptE \phi^+(x_1,x')| = t |\scriptE \phi^-(-x_1,x')|, \qquad x \in \R^{1+d}, 
 $$
where $t$ is an argument of the maximum on the right hand side of \eqref{E:def alpha}, then any sequence $(f_n)$ obeying \eqref{E:parab limit} is extremizing.
\end{theorem}

Inequality \eqref{E:op norms} is known to hold with strict inequality for the case $p=2$ in dimensions $d=1,2$, and is conjectured to be a strict inequality for $p=2$ in all dimensions.  The former result gives us the following corollary, which modestly extends the range of $p$ for which extremizers were previously known to exist for the $L^p(\S^d) \to L^{\frac{d+2}d p'}(\R^{1+d})$ extension problem.  

\begin{corollary} \label{C:d=1,2}
In dimensions $d=1,2$, for $|p-2|$ sufficiently small and $q=\tfrac{d+2}d p'$, extremizers exist for the extension operator in \eqref{E:restriction/extension}, and extremizing sequences are precompact modulo symmetries.  
\end{corollary} 

If true, the conjecture that the extremizers of the Stein--Tomas inequality for the paraboloid are Gaussians in all dimensions would imply that Corollary~\ref{C:d=1,2} holds in all dimensions \cite{FLS}.  

The question of what are these extremizers is, of course, extremely interesting, but it is beyond the scope of this article.  In \cite{ChristQuilodran}, Christ--Quilodr\'an proved that Gaussian functions are not extremal for \eqref{E:rest/extn parab} (unless, possibly, $p=2$), by proving that Gaussians do not satisfy the corresponding Euler--Lagrange equation unless $p=2$.  In the case of the sphere, however,  symmetry makes it elementary to verify that constants do satisfy the analogous Euler--Lagrange equations for all $(p,q)$, as was noted in \cite{ChristQuilodran}, but this is insufficient to verify that constants are extremizers.   

When $p=2$, Theorems~\ref{T:scaling} and \ref{T:op norms big p}, Proposition~\ref{P:op norms}, and Corollary~\ref{C:d=1,2} are due to Frank--Lieb--Sabin, \cite{FLS}.

As can be seen from the comparison with prior results, the main advantage of our approach is that it allows us to consider restriction inequalities with $p \neq 2$, for which the loss of the Hilbert space structure and Plancherel substantially reduces our available tools \cite{OpialBulletin67}.  We achieve our results by adapting the approach laid out in \cite{BSparab}, wherein it was proved that all valid, nonendpoint parabolic extension estimates possess extremizers and have precompact (modulo symmetries) extremizing sequences.  The major difference between the spherical case and the parabolic one is the defect in compactness due to the lack of a scaling symmetry.  We instead treat scaling as an almost-symmetry, analogously with prior works such as \cite{FLS, KSV}.  Relative to \cite{KSV}, the existence of distinct points on the sphere with parallel normal vectors presents an additional complication, which we address by adapting the approach of \cite{FLS}. 

\begin{figure}[htbp]
  \centering

  \begin{tikzpicture} [scale = 5]
    \fill[black!5] (0, 0) rectangle (1, 1);

\fill[green!30] (0, 0) -- (0,.375) -- (0.375, 0.375)--(1,0)-- cycle;

    \fill[green!60] (.375,.375) --  (1, 0) -- (0, 0) -- cycle;  

    \draw[dotted,thick] (.5, .3) -- (0.5, 0) node[below] {$\frac12$};
    \fill[red] (0.5, 0.3) circle [radius = 0.01];
 \draw[red] (0.375, 0.375) circle [radius = 0.01]; 
        
    \draw[->] (0, 0) -- (1.05, 0) node[right] {$\frac1p$};
    \draw[->] (0, 0) -- (0, 1.05) node[above] {$\frac1q$};
    \draw (0, 1) node[left] {$1$};
    \draw[ dotted] (1, 1) -- (0, 0) node[below left] {$0$}; 
    \draw (1, 0) node[below] {$1$};
    \draw[-] (.375, .375) -- (1, 0);  
    \draw[dotted] (.375,.375) -- (0,.375);
  \end{tikzpicture}
  
  \caption{The extension operator is conjectured to be bounded in the green quadrilateral.  We consider the subset of the darker green triangle on which the adjoint restriction conjecture holds, including the parabolic scaling line on the right, but excluding the diagonal $p=q$ on the left.  We have indicated the $p=2$ case, which featured prominently in prior work, with a dotted line. } 
  \label{fig:restriction}
\end{figure}
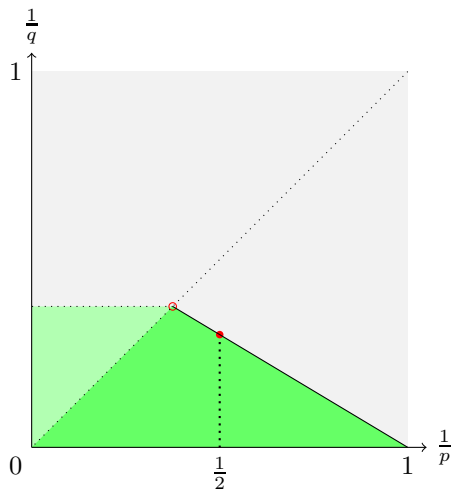

\subsection*{Outline of paper}
In the next section, we will give an in-depth overview of some of the recent history of related questions, placing our work in context.  Our strongest result, from which the others all follow, is an $L^p$-profile decomposition for bounded sequences on the sphere with nonnegligible extensions.  This result is somewhat complex, and will occupy three theorems in Section~\ref{S:profile statements}.  The first of these three results gives a frequency decomposition.  Roughly, if $\{f_n\}$ is bounded in $L^p(\S^d)$ and $\{\scriptE f_n\}$ does not tend to zero, then (along a subsequence) each $f_n$ decomposes as a finite sum of pieces with good frequency localization properties, plus a small error; this result is proved in Section~\ref{S:freq decomp}, with bilinear restriction as a primary tool.  Though the summands arising in the first decomposition are bounded by sequences that are precompact in $L^p(\S^d)$ (modulo scaling and rotation), they are not themselves precompact, as their extensions may not be well-localized in space.  The second profile decomposition establishes good spatial localization for (nearly) pointwise bounded sequences, and is proved in Section~\ref{S:scale 1 profiles}.  The third and final profile decomposition establishes good spatial localization (after rescaling) for concentrating sequences, and is proved in Section~\ref{S:concentrating profiles}.  In Section~\ref{S:off scaling}, we prove Theorem~\ref{T:off scaling}, that extremizers exist for exponent pairs $p<q$ lying off of the parabolic scaling line $q=\tfrac{d+2}dp'$.  Section~\ref{S:equal profiles} provides an analysis of the behavior of antipodally concentrating profiles, which is then applied in Section~\ref{S:scaling line} to deduce properties of concentrating extremizing sequences (supposing that they exist).  

\subsection*{Notation}  We will use throughout the standard notation $A\lesssim B$ to mean $A \leq CB$ for $C$ an \textit{admissible} constant that will be allowed to change from line to line.  Admissible constants may depend on the dimension $d$, the exponents $p,q$, and (in cases where our results are conditional) on bounds for the spherical restriction/extension operators for exponents in a small neighborhood of $p,q$. Occasionally we will decorate the `$\lesssim$' symbol with subscripts to indicate additional dependencies.  

Though we will use Lebesgue norms on three different spaces (the sphere, $\R^{d+1}$, and $\R^d$); when the meaning is clear and space is limited, we will abbreviate these norms by using only the exponent as a subscript.

\subsection*{Acknowledgements}  The authors are grateful to Arthur DressenWall for sharing some of his findings regarding the quantity $\alpha_{p \to q}$.   The second author was supported in part by NSF DMS-1653264 while working on this project.  

\section{Prior results}

An excellent survey on sharp Fourier restriction results is given in \cite{FoschiSilva}. As this is an active area, we highlight a few more recent results as well as the prior results most relevant to our analysis. For the sake of completeness, we will also state and prove an elementary result that we have not been able to find written elsewhere.

Existence results for extremizers of $L^p$-$L^q$ inequalities for the sphere have largely\footnote{A partial exception is \cite{CSSsphere} in which $p=2$ and $L^q(\R^{d+1})$ is replaced by the mixed norm space $L_{rad}^qL_{ang}^2(\R^{d+1})$, and the analysis is based upon a careful study of Bessel functions.} involved one or both of the hypotheses that $p=2$ or $q$ is an even integer. These cases are special because of the Hilbert space structure available in $p=2$, on the one hand, and an explicit formula for the $L^{2k}$ norm of spherical extensions as the $L^2$ norm of a $k$-fold convolution, on the other. In addition to the the previously discussed results of \cite{FLS} when $p=2$ and $q=\tfrac{d+2}dp'$, existence of extremizers has been established in the cases that $p=2$, $q=4$, $d=2$ \cite{ChristShao};  $p=2$, $q>\tfrac{d+2}dp'$,  $d\geq 1$ 
\cite{FVV2011};  $p=2$, $q=6$,  $d=1$ \cite{Shao16}; $p\geq 2$, $q=4$, $d\in\{2,3,4,5,6\}$;  $p \geq 4$, $q=4$, $d\geq 7$; $p\geq q$, $q=2k$, $q\geq 6$, $d\geq 1$ (the last three results are all in \cite{CS}). We note that the $p\geq q$ condition in \cite{CS} includes $p=\infty$ and is precisely the reverse of our $q>p$ condition. 

In some of these cases extremizers are known to be modulations of constants. Namely, when $p=2$, $q=4$, and $d=2$, this result is due to \cite{Foschi}; for $p\geq 2$, $q=4$, and $d\in\{2,3,4,5,6\}$,  $p \geq 4$, $q=4$, and $d\geq 7$, $p\geq q$, $q=2k$, $q\geq 6$, and $d\geq 1$, \cite{CS}; and for $p=2$, $q=2k$, when   $d\in\{2,3,4,5,6\}$ \cite{DiogoRene21}. Stability of these results is investigated in \cite{CSN}  where they show that in $d\in\{2,3,4,5,6\}$, for $p=2$, when $L^4(\R^{d})$ is replaced by a weighted $L^4$ with a radial weight which is a small perturbation of the unweighted case, the only extremizers are constants.  

Our results build on the profile decomposition approach of \cite{FVV2011} (and, implicitly, \cite{FLS}), extending these methods to address the absence of Hilbert space structure when $p \neq 2$. The methods for this adaptation originate in the study of sharp restriction for non-compact manifolds, specifically \cite{BSparab}, which proves that all valid, nonendpoint parabolic extension estimates possess extremizers and have precompact (modulo symmetries) extremizing sequences. The ideas are further developed in \cite{BiswasStovall}, \cite{tautges1}, and \cite{tautges2}   which consider the precompactness of extremizing sequences for adjoint Fourier restriction to other non-compact manifolds. Recently, $L^2$ based concentration compactness methods have also been used to investigate convergence of extremizing sequences on the hyperbola in \cite{CSS} and \cite{CSSShyp2}.

Finally, for completeness, we prove an elementary result, which is surely known to experts, but which we haven't found in the literature.  

\begin{proposition} \label{P:p infty}
For all $1 \leq p \leq \infty$, nonzero modulated constants, $e^{ix_0 \omega}\lambda$ are maximizers of the $L^p \to L^\infty$ extension inequality.  When $p > 1$, such functions are the unique maximizers.  When $1 < p < \infty$, after possible modulation and multiplication by unimodular constants, every normalized extremizing sequence in $L^p$ converges in $L^p$ to the constant function $\lambda_p:=\sigma(\S^d)^{-\frac1p}$.  
\end{proposition}

\begin{proof}
Sufficiency follows from H\"older's inequality,
\begin{equation} \label{E:p infty holder}
\|\scriptE f\|_{L^\infty(\R^{d+1})} \leq \|f\|_{L^1(\S^{d})} \leq \sigma(\S^d)^{\frac1{p'}}\|f\|_{L^p(\S^{d})},
\end{equation}
and for all $p$, equality holds for the modulated constants.  For necessity of the constants, we observe that for all $p$ the first inequality in \eqref{E:p infty holder} is equality if and only if $e^{i\theta}e^{ix_0\omega}f$ is nonnegative for some $\theta,x_0$, while for $p>1$, the second inequality is equality if and only if $|f|$ is constant.  

Finally, let $1<p<\infty$ and let $\{f_n\}$ be a normalized extremizing sequence in $L^p(\S^{d})$.  By modulating and multiplying the $f_n$ by unimodular constants, we may assume that $\scriptE f_n(0) = \|\scriptE f_n\|_{L^\infty(\R^{d+1})}$ for all $n$.  It suffices to prove that every subsequence of $\{f_n\}$ has a further subsequence convergent to $\lambda_p$.  Therefore, since $\{f_n\}$ was arbitrary, it suffices, by Banach--Alaoglu, to prove that $f_n \to \lambda_p$ in $L^p(\S^{d})$, under the additional hypothesis that $f_n$ converges weakly to some $f \in L^p(\S^{d})$.  

By construction and our above computation of the operator norm,
$$
\scriptE f(0) = \lim \scriptE f_n(0) = \lim \|\scriptE f_n\|_{L^\infty(\R^{d+1})} = \|\scriptE\|_{L^p(\S^d) \to L^\infty(\R^{d+1})} = \sigma(\S^d)^{\frac1{p'}}.
$$
Since $\|f\|_{L^p(\S^{d})} \leq 1$, while $\|\scriptE f\|_{L^\infty(\R^{d+1})} \geq \|\scriptE\|_{L^p \to L^\infty}$, the uniqueness portion of the proposition (already proved) implies that $f\equiv \lambda_p$.  Since $f_n \rightharpoonup \lambda_p$ and $\|f_n\|_{L^p(\S^{d})} \equiv 1 = \|\lambda_p\|_{L^p(\S^{d})}$, Theorem~2.11 of \cite{LiebLoss} implies that $f_n \to \lambda_p$ in ${L^p(\S^{d})}$.  
\end{proof}

Thus our setting introduces some key differences relative to what has come before.  Namely, as opposed to the vast majority of published results, we impose very few conditions on $(p,q)$, requiring only that $q>p$, $S_{p \to q}<\infty$, and that the extension conjecture is valid on a neighborhood of $(p,q)$.  Further, due to compactness of the sphere, we are able to consider an even wider range of exponent pairs than \cite{BSparab} (which was limited to the scaling line).  Additionally, as already observed in \cite{FLS}, the sphere lacks some simplifications available for other surfaces (e.g.\ the paraboloid or hyperboloid), since it lacks a scaling symmetry and possesses antipodal points; consideration of these features without the condition $p=2$ presents some new complications.  Finally, our results go further than those of \cite{BiswasStovall, BSparab, tautges1, tautges2} by establishing a full profile decomposition for bounded $L^p$ sequences, rather than exclusively focusing on the extremal case. 

\section{A weak $L^p$ profile decomposition} \label{S:profile statements}

In this section, we will introduce our main tool, a weak $L^p$ profile decomposition.  A profile decomposition associated to an operator $T:X \to Y$ is a means of decomposing bounded sequences in $X$ as the sum of a structured part, which has good compactness properties modulo symmetries, and a ``random'' part, which is small after an application of $T$.  The method was introduced by Lions \cite{Lions} and has found extensive application in PDE.  For Fourier extension operators, the $L^2$-based theory of  profile decompositions is comparatively well-developed, both because of the role that $L^2$-based inequalities play in the study of dispersive equations and also because more tools, namely, Plancherel and the Hilbert space structure, are available.  In particular, all of the essential ingredients for the $L^2$ profile decomposition are given in \cite{FLS}, though the full profile decomposition was never explicitly stated in that article.  

In \cite{BSparab}, a profile decomposition of extremizing, frequency localized $L^p$ sequences was used to prove that the extension operator associated to the paraboloid possesses extremizers.  Here, we give a more quantitative result, providing a decomposition of more general sequences.  Our profile decomposition is weak in the sense that it gives poor control over the remainder terms, which, despite having small extension, may blow up in $L^p$.  This blowup of the $L^p$ norm results from our use of weak limits, and does not affect the $L^2$ theory because of elementary Hilbert space manipulations.  (When we are not in a Hilbert space, subtracting a weak limit from a sequence does not necessarily decrease the limit of the norms \cite{OpialBulletin67}.)  An alternative, stronger profile decomposition for $L^p$ sequences and operators satisfying certain conditions is developed in \cite{SoliminiTintarev}; it is based on $\Delta$-limits, rather than weak limits.  A significant advantage of using $\Delta$-limits, rather than weak limits, is that the remainder terms in the $\Delta$ profile decomposition are bounded in $L^p$, in addition to having small extensions.  A disadvantage is that $\Delta$-limits do not seem to yield sufficiently sharp inequalities to control the number of profiles of an extremizing sequence and thereby prove the existence of extremizers.  (I.e., we will rely on  inequalities involving the relation `$\leq$,' rather than `$\lesssim$.')

Our results include the possible case of concentration at antipodal points.  For this reason, it is convenient to use the real projective space $\R\P^d = \S^d/\{\omega \sim -\omega\}$, whose elements we denote by $[\omega] := \{\omega,-\omega\}$, $\omega \in \S^d$.  We observe that 
$$
\dist([\omega],[\omega']) =\min\{|\omega-\omega'|, |\omega+\omega'|\}.
$$

To produce statements that are somewhat easier to parse, we have broken our profile decomposition into three parts.  We begin with a decomposition of the frequency space $\S^d$, distinguishing between the critical and subcritical regime. 

\begin{theorem}[Frequency decomposition]\label{T:freq}
Let $1<p<q=\tfrac{d+2}dp'$, and assume that the restriction conjecture for $\scriptE$ holds on a neighborhood of $(p,q)$.  
Let $\{f_n\}$ be a bounded sequence in $L^p(\S^d)$.  After passing to a subsequence, there exist $\{\lambda_n^j\}_{j,n \in \N} \subseteq (0,1]$, $\{[\omega_n^j]\}_{j,n \in \N} \subseteq \R\P^d$, and  a sequence of decompositions $f_n = \sum_{j=1}^J F_n^j + R_n^J$, $J \in \N$, such that:\\
\emph{(i)} For each $j$, either $\lambda_n^j \to 0$, or $\lambda_n^j \equiv 1$.  \\
\emph{(ii)} For $j \neq j'$, either $|\log\tfrac{\lambda_n^j}{\lambda_n^{j'}}| \to \infty$ or $\lambda_n^j \equiv \lambda_n^{j'}$, and, in the latter case, if $\lambda_n^j \to 0$,  then $(\lambda_n^j)^{-1} \dist([\omega_n^j],[\omega_n^{j'}]) \to \infty$.  \\
\emph{(iii)}  For each $n,J\in \N$, $\|f_n\|_{L^p(\S^{d})}^p = \sum_{j=1}^J\|F_n^j\|_{L^p(\S^{d})}^p + \|R_n^J\|_{L^p(\S^{d})}^p$,\\
\emph{(iv)} For all $J\in \N$, $\lim_{n \to \infty} \|\scriptE f_n\|_q^q - \sum_{j=1}^J \|\scriptE F_n^j\|_q^q - \|\scriptE R_n^J\|_q^q = 0$, \\
\emph{(v)} The remainders have small extension:  $\lim_{J \to \infty} \limsup_{n \to \infty} \|\scriptE R_n^J\|_{L^q(\R^{d+1})} = 0$\\
\emph{(vi)} The $F_n^j$ are adapted to antipodal caps of radius $\lambda_n^j$ with centers on $[\omega_n^j]$ in the sense that 
\begin{gather*}
    \lim_{M \to \infty}\limsup_{n \to \infty} \sup_{E \subseteq E_n^{j,M}} \|\scriptE F_n^j \chi_E\|_{L^q(\R^{d+1})} = 0, \qtq{where}\\
    E_n^{j,M} := \{|F_n^j|>M(\lambda_n^j)^{-d/p}\} \cup \{\dist(\omega,[\omega_n^j])>M\lambda_n^j\}.    
\end{gather*}
\end{theorem}

We observe that without (vi), the result follows trivially by taking $F_n^1 \equiv f_n$.  

In the subcritical regime $q > \tfrac{d+2}d p'$, the frequency ``decomposition'' is much simpler (and an elementary consequence of H\"older's inequality, as we will see).  

\begin{proposition} \label{P:subcrit freq}
Let $q>\max\{\tfrac{d+2}dp',\tfrac{2(d+1)}d\}$, and assume that the restriction conjecture for $\scriptE$ holds on a neighborhood of $(p,q)$.  
Let $\{f_n\}$ be a bounded sequence in $L^p(\S^d)$.  Then if $E_n^M:=\{|f_n|>M\}$, then
$$
\lim_{M \to \infty} \sup_n \sup_{E \subseteq E_n^M} \|\scriptE f_n\chi_E\|_{L^q(\R^{d+1})} = 0.  
$$
\end{proposition}

Our next two results provide a spatial decomposition of functions obeying the frequency localization property described in part (vi) of Theorem~\ref{T:freq} (and the conclusion of Proposition~\ref{P:subcrit freq}), in the cases of nonconcentration and antipodal concentration, respectively.  In both, we will use the notation
$$
\tilde p:=\max\{p,p'\}.
$$

\begin{theorem}[Scale 1 spatial decomposition]\label{T:space1}
Let $1<p<q<\infty$ with $q\geq\tfrac{d+2}dp'$ and $q>\tfrac{2(d+1)}d$, and assume that the restriction conjecture for $\scriptE$ holds on a neighborhood of $(p,q)$.  
Let $\{f_n\}$ be a bounded sequence in $L^p(\S^d)$ satisfying the condition
$$
\lim_{M \to \infty} \limsup_{n \to \infty} \|\scriptE f_n\chi_{\{|f_n|>M\}}\|_{L^q(\R^{d+1})} = 0.
$$
After passing to a subsequence, there exist $\{x_n^j\}_{j,n \in \N} \subseteq \R^{d+1}$ obeying
\begin{equation} \label{E:xnj-xnj' scale 1}
\lim_{n \to \infty}|x_n^j-x_n^{j'}| = \infty, \qtq{for} j \neq j',
\end{equation}
and weak limits, $\phi^j = \wklim e^{-ix_n^j\omega}f_n$, such that for $J \in \N$,
\begin{equation} \label{E:Lp orthog 1}
\begin{gathered}
\bigl(\sum_j \|\phi^j\|_{L^p(\S^d)}^{\tilde p}\bigr)^{\frac1{\tilde p}} \leq \liminf\|f_n\|_{L^p(\S^d)}, \\ \limsup\|\sum_{j=1}^{J} e^{ix_n^j\omega}\phi^j\|_{L^p(\S^d)} \leq \bigl(\sum_{j=1}^J \|\phi^j\|_{L^p(\S^d)}^{\tilde p'}\bigr)^{\frac1{\tilde p'}}, 
\end{gathered}
\end{equation}  
and the remainders $r_n^J:=f_n-\sum_{j=1}^J e^{ix_n^j\omega}\phi^j$ satisfy
\begin{gather}\label{E:Lq orthog 1}
\lim_{n \to \infty} \|\scriptE f_n\|_{L^q(\R^{d+1})}^q - \sum_{j=1}^J \|\scriptE \phi^j\|_{L^q(\R^{d+1})}^q - \|\scriptE r_n^J\|_{L^q(\R^{d+1})}^q = 0, \quad J \in \N,\\ \label{E:ErnJ 0 1}
\lim_{J \to \infty} \limsup_{n \to \infty}\|\scriptE r_n^J\|_{L^q(\R^{d+1})} = 0.
\end{gather}
\end{theorem}

In what follows, we write $\omega=(\omega_1,\omega')$. 

\begin{theorem}[Large scale spatial decomposition] \label{T:large space}
Let $1<p<\tfrac{2(d+1)}d$ and $q=\tfrac{d+2}dp'$. Assume that the restriction conjecture for $\scriptE$ holds on a neighborhood of $(p,q)$.  
Let $\{f_n\} \subset L^p(\S^d)$ be a bounded sequence and let $\lambda_n\searrow 0$.  Assume that 
\begin{equation} \label{E:Efn>M large space}
\begin{gathered}
    \lim_{M \to \infty}\limsup_{n \to \infty} \sup_{E \subseteq E_n^{M}} \|\scriptE f_n \chi_E\|_{L^q(\R^{d+1})} = 0, \qtq{where}\\
    E_n^{M} := \{|f_n|>M(\lambda_n)^{-d/p}\} \cup \{\dist(\omega,[e_1])>M\lambda_n\}
\end{gathered}
\end{equation}
After passing to a subsequence, there exist $\{x_n^j\}_{j,n \in \N} \subset \R^{d+1}$ with
$$
\lim_{n \to \infty}\bigl( \lambda_n^2|(x_n^j-x_n^{j'})_1|+\lambda_n|(x_n^j-x_n^{j'})'|\bigr) = \infty, \qtq{for} j \neq j',
$$
 and weak limits $\phi^{j,\bullet} \in L^p(\R^d)$, $\bullet = +,-$, given by
 \begin{equation} \label{E:phijpm}
\phi^{j,\pm} = \wklim \lambda_n^{d/p}e^{-ix_n^j(\pm\sqrt{1-|\lambda_n\xi|^2},\lambda_n\xi)}f_n(\pm\sqrt{1-|\lambda_n\xi|^2},\lambda_n\xi)\chi_{\{|\xi|<\frac12 \lambda_n^{-1}\}},
\end{equation}
such that the following conclusions hold.  Setting 
\begin{equation}\label{E:gnj def} 
    g_n^{j,\pm}(\omega):=   \lambda_n^{-d/p}\phi^{j,\pm}(\lambda_n^{-1}\omega')\chi_{\{\pm\omega_1>0\}}\chi_{\{|\omega'|<\tfrac12\}}, \qquad g_n^j:=\sum_{\pm} g_n^{j,\pm},
\end{equation}
and 
\[ 
r_n^J:=f_n-\sum_{j=1}^J e^{ix_n^j \omega}g_n^{j},\] 
then,\\
\emph{(i)} $\lim_{n\to \infty} \| \scriptE g_n^j -\sum_\pm \lambda_n^{\frac{d+2}q}e^{\pm ix_1}\scriptE_\P \phi^{j,\pm}(\mp\lambda_n^2x_1,\lambda_n x') \|_{L^q(\R^{d+1})} = 0.$  \\
\emph{(ii)} $\bigl[\sum_\pm (\sum_j\|\phi^{j,\pm}\|_{L^p(\R^{d})}^{\tilde p}\bigr)^{p/\tilde p}\bigr]^{1/\tilde p} \leq \liminf \|f_n\|_{L^p(\S^{d})}$,\\
\emph{(iii)} $\limsup_{n \to \infty}\|\sum_{j=1}^J e^{ix_n^j\omega}g_n^j\|_{L^p(\S^{d})} \leq \bigl[\sum_\pm \bigl(\sum_j\|\phi^{j,\pm}\|_{L^p(\R^{d})}^{\tilde p'}\bigr)^{p/\tilde p'}\bigr]^{1/\tilde p'}$, $J \in \N$,\\
\emph{(iv)} $\lim_{n \to \infty} \|\scriptE f_n\|_{L^q(\R^{d+1})}^q - \sum_{j=1}^J\|\scriptE g_n^j\|_{L^q(\R^{d+1})}^q - \|\scriptE r_n^J\|_{L^q(\R^{d+1})}^ q = 0$, $J \in \N$\\
\emph{(v)} $\lim_{J \to \infty}\limsup_{n \to \infty} \|\scriptE r_n^J\|_{L^q(\R^{d+1})} = 0$.
\end{theorem}

\section{Proof of the frequency decomposition} \label{S:freq decomp}

In this section, we will prove Theorem~\ref{T:freq} and Proposition~\ref{P:subcrit freq}, which decompose the functions from bounded sequences into pieces with good frequency localization.  

We begin with the essentially elementary proof of  Proposition~\ref{P:subcrit freq}, in which $(p,q)$ lies off of the parabolic scaling line.  

\begin{proof}[Proof of Proposition~\ref{P:subcrit freq}]
By assumption, $\scriptE$ extends as a bounded operator from $L^r(\S^d)$ to $L^s(\R^{d+1})$ for $(r,s)$ in a neighborhood of $(p,q)$.  In particular, $\scriptE$ maps $L^r(\S^d)$ into $L^q(\R^{d+1})$ for some $r<p$.  For $f \in L^p(\S^d)$, $M>0$, and $E\subseteq E^M=\{
|f|>M\}$ which is a measurable set,
$$
\|\scriptE (f\chi_E)\|_{L^q(\R^{d+1})} \lesssim \|f \chi_E\|_{L^r(\S^d)}  \leq M^{-(\frac pr-1)}\|f\|_{L^p(\S^d)}^{\frac pr}.
$$
Hence for a bounded sequence $\{f_n\}$, the stated conclusion holds.
\end{proof}

The frequency localization is more involved for exponents on the scaling line.  The heart of the argument is the bilinear extension theorem of Tao \cite{TaoParab} and the bilinear-to-linear argument of Tao--Vargas--Vega \cite{TVV}, with inspiration from \cite{BV, Bourgain98, CarlesKeraani}.

\begin{lemma}[\cite{TaoParab}] \label{L:bilin}
Let $1<p<\tfrac{2(d+1)}d$ and $q=\tfrac{d+2}d p'$, and assume that the spherical extension conjecture holds on a neighborhood of $(p,q)$.  Then there exists $s < p$ such that for all $(\bar s,\bar q)$ in a neighborhood of $(s,q)$, all $0 < r \ll 1$, and $f \in L^{\bar s}$, we have the bilinear inequality
\begin{equation}\label{E:bilin}
\|\scriptE (f \chi_\tau) \scriptE (f\chi_{\tau'})\|_{\frac{\bar q}2} \lesssim r^{-2d(\tfrac{d+2}{d\bar q}-\tfrac1{\bar s'})}\|f\chi_\tau\|_{L^{\bar s}(\S^d)}\|f\chi_{\tau'}\|_{L^{\bar s}(\S^d)},
\end{equation}
whenever $\tau,\tau'$ are defined by
$$
\tau:=\{\nu \in S^d : \dist([\nu],[\omega]) \ll r\}, \qquad \tau':=\{\nu \in \S^d : \dist([\nu],[\omega']) \ll r\},
$$
for some $\omega,\omega' \in \S^d$ with $\dist([\omega],[\omega']) \sim r$.  The implicit constant is independent of $f,\tau,\tau',r$.  
\end{lemma}

We give the details of the deduction from the remarks in Section~9 of \cite{TaoParab}.  

\begin{proof}  Since exponents $\bar q$ lying sufficiently close to $q$ have similar properties (namely, finiteness of $S_{\bar p \to \bar q}$ when $\bar p = (\tfrac{2q}{d+2})'$), it suffices prove that such an estimate holds at $(s,q)$, for $s$ lying in some open subinterval of $(1,p)$.

After a rotation, we may assume that $\omega = e_1$.  We decompose $\tau = \tau_0 \cup (-\tau_0)$, $\tau' = \tau_0' \cup (-\tau_0')$, where $\diam(\tau_0) \sim \diam(\tau_0') \ll r \sim \dist(\tau_0,\tau_0')$.  Using the triangle inequality and taking conjugates, it suffices to bound
$$
\|\scriptE(f\chi_{\tau_0})\,\scriptE(f\chi_{\tau'_0})\|_{L^{\frac q2}(\R^{d+1})} .
$$ 

By assumption, for $\bar q$ sufficiently near $q$ and $\bar p:=(\tfrac{d\bar q}{d+2})'$, we have the linear estimate $S_{\bar p \to \bar q}<\infty$.  Hence by Cauchy--Schwarz, for any $f_1$, $f_2$: 
\begin{equation} \label{E:CS 1 bilin}
\|\scriptE f_1 \scriptE f_2 \|_{L^{\frac{\bar q}2}} \leq \|\scriptE f_1\|_{L^{\bar q}}\| \scriptE f_2 \|_{L^{\bar q}} \leq S_{\bar p \to \bar q}^2 \|f_1\|_{L^{\bar p}(\S^d)}\| f_2 \|_{L^{\bar p}(\S^d)}.
\end{equation}

Further Tao's bilinear estimate \cite{TaoParab} gives that for any $t>\frac{2(d+3)}{d+1}$ and any pair of test functions $f_1$, $f_2$ supported on spherical caps whose width and separation are comparable to some sufficiently small dimensional constant, 
\begin{equation} \label{E:Tao 1 bilin}
\|\scriptE f_1 \scriptE f_2 \|_{L^{\frac t2}(\R^{d+1})} \leq C \|f_1\|_{L^2(\S^{d})}\| f_2 \|_{{L^2(\S^{d})}}.
\end{equation}

Hence by interpolating \eqref{E:Tao 1 bilin} for some $\tfrac{2(d+3)}{d+1} < t < \tfrac{2(d+2)}d$ with \eqref{E:CS 1 bilin}, we see that \eqref{E:bilin} holds in the case $r\sim 1$.  Inequality \eqref{E:bilin} in the case $r \ll 1$ follows by parabolic rescaling; we omit the details.  
\end{proof}

For technical reasons, we will use a slightly non-typical definition of caps.  By a \textit{(square) cap} in $\S^d$, we mean the intersection of $\S^d$ with with two axis-parallel cubes having diameter at most $\tfrac14$, with centers given by two antipodal points contained in $S^d$.  By the sidelength of a cap, we mean the sidelength of the corresponding cube.   For each $j\in\{1,2\ldots,d\}$ and each $k \in \N$ satisfying $k \geq C$ for some sufficiently large $C$, we fix a non-overlapping covering $\scriptD^k_j$ of the nonequatorial region $W_j:= \{\omega \in S^d:|\omega_j| \geq \tfrac1{2\sqrt d}\}$ by caps  of sidelength $2^{-k}$.  We denote various unions of the $\scriptD^k_j$ by
$$
\scriptD^k:=\bigcup_{j=1}^d \scriptD^k_j, \qquad \scriptD_j:=\bigcup_{k \geq C}\scriptD^k_j, \qquad \scriptD:=\bigcup_{j=1}^d \scriptD_j.
$$
We say that two caps $\tau,\tau' \in \scriptD^k_j$ are related, $\tau \sim \tau'$, if $2^{-k+C} \leq \dist(\tau,\tau') \leq 2^{-k+ 2C}$.  It is well-known (see \cite{TVV, BV}) that when $C$ is sufficiently large, all of the possible sumsets $\tau+\tau'$ for related $\tau,\tau' \in \scriptD^k, k\geq C$ are contained in finitely overlapping parallelepipeds.  (These parallelepipeds have much smaller volume than the cubes whose intersection with $S^d$ equals $\tau,\tau'$.)  We use $f_\tau$ to denote the product $f \cdot \chi_\tau$.

\begin{lemma} \label{L:positive bound}
Let $q > \frac{2(d+1)}d$ and $p = (\frac{qd}{d+2})'$, and assume that the spherical extension conjecture holds on a neighborhood of $(p,q)$.  There exists $s<p$ such that for sufficiently small $0<\nu<1$,  
\begin{equation} \label{E:positive bound}
\|\scriptE f\|_q \lesssim (\sup_{\tau \in \scriptD}|\tau|^{-\frac1{p'}}\|\scriptE f_\tau\|_\infty)^\nu \bigl(\sum_{\tau \in \scriptD} |\tau|^{-2t(1-\nu)(\frac1s-\frac1p)}\|f_\tau\|_{L^s(\S^d)}^{2t(1-\nu)}\bigr)^{\frac1{2t}}.
\end{equation}
Here $t := \min\{\frac q2,(\frac q2)'\}$.
\end{lemma}

\begin{proof}
By Lemma~\ref{L:bilin}, there exists $s<p$ such that for sufficiently small $0<\nu<1$ and $\bar q := (1-\nu)q$, \eqref{E:bilin} holds with exponent $\frac{\bar q}2$ on the left side and exponent $s$ on the right.

By the triangle inequality, there exists $j$ such that $\|\scriptE f\|_q \lesssim \|\scriptE (f\chi_{W_j})\|_q$, and after a rotation, we may assume that $j=1$.  Employing a Whitney decomposition of $W_1 \times W_1 \setminus\{(\omega,\omega):\omega \in S^d\}$, followed by almost orthogonality (Lemma~6.1 of \cite{TVV}), H\"older's inequality, and finally our bilinear extension estimate, the arithmetic-geometric mean inequality, and some reindexing,
\begin{align*}
\|\scriptE (f\chi_{W_1})\|_q^2 &=  \|\sum_{\tau \sim \tau' \in \scriptD_1} c_{\tau,\tau'}\scriptE f_\tau \scriptE f_{\tau'}\|_{\frac q2}
\lesssim \bigl(\sum_{\tau \sim \tau' \in \scriptD}\|\scriptE f_\tau \scriptE f_{\tau'}\|_{\frac q2}^t\bigr)^{\frac 1t}\\
&\leq \bigl(\sum_{\tau \sim \tau' \in \scriptD}(|\tau|^{-\frac2{p'}}\|\scriptE f_\tau\scriptE f_{\tau'}\|_\infty)^{t\nu}(|\tau|^{\frac{2\nu}{(1-\nu)p'}}\|\scriptE f_\tau \scriptE f_{\tau'}\|_{\frac{\bar q}2})^{t(1-\nu)}\bigr)^{\frac1t}\\
& \lesssim 
(\sup_{\tau \in \scriptD} |\tau|^{-\frac1{p'}}\|\scriptE f_\tau\|_\infty)^{2\nu}
\bigl(\sum_{\tau \in \scriptD} (|\tau|^{-(\frac1s-\frac1p)}\|f_\tau\|_{L^s(\S^d)})^{2t(1-\nu)} \bigr)^{\frac1t}.
\end{align*}
Here the $c_{\tau,\tau'}$ are constants with $|c_{\tau,\tau'}| \lesssim 1$.  
\end{proof}

\begin{lemma} \label{L:scalable}
Let $q=\tfrac{d+2}d p' > p > 1$, and assume that the spherical extension conjecture holds on a neighborhood of $(p,q)$.  Then there exist $c_0 > 0$ and $0 < \theta < 1$ such that
\begin{equation} \label{E:scalable}
\|\scriptE f\|_q \lesssim (\sup_{\tau \in \scriptD}|\tau|^{-\frac1{p'}}\|\scriptE f_\tau\|_\infty)^\theta \|f\|_p^{1-\theta} \lesssim \sup_{\tau \in \scriptD} \sup_{n\geq0} 2^{-c_0n}\|f_{\tau,n}\|_p^\theta\|f\|_p^{1-\theta},
\end{equation}
where $f_{\tau,n}$ equals $f$ multiplied by the characteristic function of $\tau \cap \{|f| \leq 2^n\|f\|_p |\tau|^{-\frac1p}\}$.
\end{lemma}

\begin{proof}  
We will prove the superficially stronger bound wherein $f_{\tau,n}$ is replaced by $f_\tau^n$ on the right hand side, where $f_\tau^0:=f_{\tau,0}$ and $f_\tau^n:=f_{\tau,n}-f_{\tau,n-1}$ for $n \geq 1$.  We observe that for $n \geq 1$, 
\begin{equation} \label{E:f tau n sim}
|f_\tau^n| \sim 2^n |\tau|^{-\frac1p}\|f\|_{L^p(\S^d)}\chi_{\{f_\tau^n \neq 0\}}.
\end{equation}

We begin by showing the second inequality in \eqref{E:scalable}.  We apply H\"older's inequality and decompose into the $f_\tau^n$ to see that
\begin{equation} \label{E:f tau n L1}
|\tau|^{-\frac1{p'}}\|\scriptE f_\tau\|_\infty \leq \sum_{n=0}^\infty |\tau|^{-\frac1{p'}} \int_{\S^d} |f_\tau^n|\, d\sigma.
\end{equation}
By H\"older's inequality, $|\tau|^{-\frac1p}\|f_\tau^0\|_{L^1} \lesssim \|f_\tau^0\|_{L^p}$, and by basic arithmetic and \eqref{E:f tau n sim}, 
$$
|\tau|^{-\frac1{p'}}\int_{\S^d}|f_\tau^n|\, d\sigma \lesssim 2^{-n(p-1)}\|f\|_p^{-(p-1)}\|f_\tau^n\|_p^p \leq 2^{-n(p-1)}\|f_\tau^n\|_p.
$$
Inserting these estimates into \eqref{E:f tau n L1}, using H\"older's inequality, and summing a geometric series,
$$
|\tau|^{-\frac1{p'}}\|\scriptE f_\tau\|_\infty \lesssim \sup_{n \geq 0} 2^{-c_0 n}\|f_\tau^n\|_p,
$$
whenever $c_0<p-1$.  

Now we turn to the first inequality in \eqref{E:scalable}.  By \eqref{E:positive bound}, it suffices to prove that 
$$
\sum_{\tau \in \scriptD} |\tau|^{-2t(1-\nu)(\frac1s-\frac1p)}\|f_\tau\|_s^{2t(1-\nu)} \lesssim \|f\|_p^{2t(1-\nu)}.
$$

Since $2t>p$, for $\nu$ sufficiently small, $2t(1-\nu)>p>s$.  Hence by the triangle inequality and H\"older's inequality, 
\begin{align} \label{E:sum f taus}
\sum_{\tau \in \scriptD}|\tau|^{-2t(1-\nu)(\frac1s-\frac1p)}\|f_\tau\|_s^{2t(1-\nu)}
\lesssim& \sum_{\tau \in \scriptD}|\tau|^{-2t(1-\nu)(\frac1{2t(1-\nu)}-\frac1p)}\|f_\tau^0\|_{2t(1-\nu)}^{2t(1-\nu)}\\\notag
&+ (\sum_{\tau \in \scriptD}|\tau|^{-2t(1-\nu)(\frac1s-\frac1p)}\|f_\tau-f_\tau^0\|_s^s)^{\frac{2t(1-\nu)}s}.
\end{align}
We recall that $\scriptD = \bigcup_{k \geq 0} \scriptD_k$ and that each $\scriptD_k$ is a finitely overlapping cover of $S^d$ by caps $\tau$ of measure $2^{-kd}$.  Again using the fact that $s < 2t(1-\nu)$, we may bound the right hand side of \eqref{E:sum f taus} by a constant multiple of 
\begin{align*}
&\sum_k 2^{-2kdt(1-\nu)(\frac1p-\frac1{2t(1-\nu)})} \int_{\{|f|\lesssim2^{\frac{kd}p}\|f\|_p\}} |f|^{2t(1-\nu)}\,d\sigma\\
&\qquad + (\sum_k 2^{kds(\frac1s-\frac1p)}\int_{\{|f| \gtrsim 2^{\frac{kd}p}\|f\|_p\}} |f|^s d\sigma)^{\frac{2t(1-\nu)}s}.
\end{align*}
By Fubini, the preceding sum equals 
\begin{align*}
&\int |f|^{2t(1-\nu)} \bigl(\sum_{k:|f|\lesssim2^{\frac{kd}p}\|f\|_p}2^{-2kdt(1-\nu)(\frac1p-\frac1{2t(1-\nu)})}\bigr)\,d\sigma\\
&\qquad + \left(\int |f|^s \bigl(\sum_{k:|f| \gtrsim 2^{\frac{kd}p}\|f\|_p}2^{kds(\frac1s-\frac1p)}\bigr)\,d\sigma\right)^{\frac{2t(1-\nu)}s},
\end{align*}
and we conclude by summing the geometric series.
\end{proof}

\begin{proof}[Proof of Theorem~\ref{T:freq} for $q=\tfrac{d+2}d p'$]
Multiplying by a constant and passing to a subsequence, we may assume that $\|f_n\|_p \to 1$ and that $\|f_n\|_p \leq 2$ for all $n$. We begin by decomposing the $f_n$ into \textit{chips} with good frequency/$L^p$ orthogonality (but whose extensions are not orthogonal in space/$L^q$).  Namely, we set $r_n^0:=f_n$ and   
$$
f_n = \sum_{j=1}^J h_n^j + r_n^J, \qquad 0 \leq J < \infty,
$$
with $h_n^j = r_n^{j-1}\chi_{\tau_n^j}\chi_{\{|f_n|\leq2^j |\tau_n^j|^{-\frac1p}\}}$.  Here $\tau_n^j$ is a dyadic spherical cap, with center $\nu_n^j$ and diameter $\rho_n^j$, that is chosen to maximize $\|h_n^j\|_p$.  An  application of the dominated convergence theorem shows that there is indeed such a maximal cap for each $j,n$ and that for each $n$, $\sum_j h_n^j$ converges to $f_n$ in $L^p(\S^d)$.  

The next lemma is the key step to establish conclusion (vi) of the theorem.  

\begin{lemma} \label{L:freq remainders 0}
Under the hypotheses of Theorem~\ref{T:freq}, remainders in the decomposition above satisfy
$$
\lim_{J \to \infty} \sup_{n \in \N} \sup_{|R_n^J| = |r_n^J|\chi_E} \|\scriptE R_n^J\|_q = 0,
$$
where the supremum is taken over measurable functions $R_n^J$ whose absolute value equals a characteristic function times $|r_n^J|$.   
\end{lemma}

\begin{proof}[Proof of Lemma~\ref{L:freq remainders 0}]
Let $J \geq 1$, $n \in \N$, and $|R_n^J| = |r_n^J|\chi_E$.  We claim that for any dyadic spherical cap $\tau$ and $j < J$,
\begin{equation} \label{E:rnJ on tau}
\|R_n^J \chi_\tau \chi_{\{|R_n^J| \leq 2^j\|R_n^J\|_p|\tau|^{-\frac1p}\}}\|_{L^p(\S^d)} \leq \tfrac2{(J-j)^{\frac1p}}.
\end{equation}
Indeed, if \eqref{E:rnJ on tau} holds, then taking $j \sim \tfrac J2$ and applying Lemma~\ref{L:scalable} completes the proof of the lemma.

 Now we turn to \eqref{E:rnJ on tau}.  By the construction of each $r_n^i$ from $r_n^{i-1}$, there exist measurable sets $E_n^0 \supseteq \cdots \supseteq \ldots\supseteq E_n^{J}=E$, such that $|R_n^J| = |r_n^i|\chi_{E_n^i}$, $i=0,\ldots,J$.  We also recall that $r_n^0 = f_n$.  We then have
$$
|R_n^J| \chi_\tau \chi_{\{|R_n^J| \leq 2^j\|R_n^J\|_p|\tau|^{-\frac1p}\}} \leq |r_n^{i-1}| \chi_\tau \chi_{\{|f_n| \leq 2^i \|f_n\|_p|\tau|^{-\frac1p}\}}, \qquad i=j,\ldots,J+1.
$$
Since each $\tau_n^i$ is chosen to maximize $\|h_n^i\|_p$, if \eqref{E:rnJ on tau} were to fail, we would also have $\|h_n^i\|_p > \tfrac2{(J-j)^{\frac1p}}$ for $i=j,\ldots,J+1$, whence $\|f_n\|_p^p \geq \sum_{i=j}^{J+1}\|h_n^i\|_p^p > 2$, a contradiction.  This completes the proof of \eqref{E:rnJ on tau}, and hence the lemma.
\end{proof}

Our next step is to organize the chips into clumps, which also possess good spatial orthogonality.  After passing to a successively passing to a subsequence for each $j$ and then choosing a diagonal subsequence, we may assume the existence of all of the limits arising below.  We form a partition $\N = \bigcup_{i=1}^{\infty} \scriptJ^i$ (the $\scriptJ^i$ may be empty for large $i$) so that $j$ and $j'$ lie in the same $\scriptJ^i$ if and only if the conditions $\lim_{n \to \infty} \tfrac{\rho_n^j}{\rho_n^{j'}} \in (0,\infty)$ and $\{(\rho_n^j)^{-1} \dist([\nu_n^j],[\nu_n^{j'}])\}_{n \in \N}$ is bounded, both hold.  Thus for $i \neq i'$, $j \in \scriptJ^i, j' \in \scriptJ^{i'}$ implies 
$$
\lim_{n \to \infty}\tfrac{\rho_n^j}{\rho_n^{j'}} \in \{0,\infty\}; \qtq{or} \lim_{n \to \infty} \tfrac{\rho_n^j}{\rho_n^{j'}} \in (0,\infty), \qtq{and} \lim_{n \to \infty} (\rho_n^j)^{-1}\dist([\nu_n^j],[\nu_n^j]) = \infty.
$$

Let $F_n^i := \sum_{j \in \scriptJ^i} h_n^j$ and $R_n^I := f_n-\sum_{i=1}^I F_n^i$; for each $n$, these sums converge in $L^p$ by the dominated convergence theorem. Passing to a further subsequence, we may associate to each $\scriptJ^i$ sequences $\{\lambda_n^i\}$, $\{[\omega_n^i]\}$ satisfying, for each $i$: $\lambda_n^i \to 0$ or $\lambda_n^i \equiv 1$, and $[\omega_n^i] \equiv [e_1]$ if $\lambda_n^i \equiv 1$; and $\lim \tfrac{\rho_n^j}{\lambda_n^i} \in (0,\infty)$, and $(\lambda_n^i)^{-1}\dist([\omega_n^i],[\nu_n^j])$ is bounded, for $j \in \scriptJ^i$.  

Conditions (i-iii) of Theorem~\ref{T:freq} are automatic, while (v-vi) are simple consequences of the construction and Lemma~\ref{L:freq remainders 0}.  It remains to prove (iv), which asserts that the extensions of the $F_n^i$ have good $L^q$ orthogonality.  

We define
$$
f_n^{\leq J}:=f_n-r_n^J, \qquad F_n^{i,\leq J}:=\sum_{j \in \scriptJ^i, j \leq J} h_n^j, \qquad R_n^{I,\leq J}:=f_n^{\leq J}-\sum_{i=1}^I F_n^{i,\leq J}.
$$
By Lemma~\ref{L:freq remainders 0}, for each $I$, 
$$
\lim_{n \to \infty} \|\scriptE f_n\|_q^q - \sum_{i=1}^I \|\scriptE F_n^i\|_q^q - \|\scriptE R_n^I\|_q^q = \lim_{J \to \infty}\lim_{n \to \infty} \|\scriptE f_n^{\leq J}\|_q^q - \sum_{i=1}^I \|\scriptE F_n^{i,\leq J}\|_q^q - \|\scriptE R_n^{I,\leq J}\|_q^q,
$$
so to prove (iv), it suffices to prove that for all $I$ and $J$,
$$
\lim_{n \to \infty} \|\scriptE f_n^{\leq J}\|_q^q - \sum_{i=1}^I \|\scriptE F_n^{i,\leq J}\|_q^q - \|\scriptE R_n^{I,\leq J}\|_q^q = 0.  
$$

The elementary inequality
$$
\left| |\sum_{i=1}^I x_i|^q - \sum_{i=1}^I |x_i|^q \right| \leq C_{q,I} \left(\sup_{i \neq j} |x_i||x_j|^{q-1}\right), \qquad q \geq 2,
$$
H\"older's inequality, boundedness of the $\|f_n\|_p$, and the triangle inequality imply that 
$$
\left|\|\scriptE f_n^{\leq J}\|_q^q - \sum_{i=1}^I \|\scriptE F_n^{i,\leq J}\|_q^q - \|\scriptE R_n^{I,\leq J}\|_q^q\right| \leq C_{I,J,q} \sup_{i \neq i', j \in \scriptJ^i, j' \in \scriptJ^{i'}} \|\scriptE h_n^j \scriptE h_n^{j'}\|_{\frac q2} .  
$$

Thus it remains to prove that for $i \neq i'$, $j \in \scriptJ^i$, and $j' \in \scriptJ^{i'}$,
\begin{equation} \label{E:Lq bilin decoup}
\lim_{n \to \infty} \|\scriptE h_n^j \scriptE h_n^{j'}\|_{\frac q2} = 0.  
\end{equation}
Suppose first that $\tfrac{\rho_n^j}{\rho_n^{j'}} \to \infty$.  By assumption, there exist $q_0 < q < q_1$, $p_0 > p > p_1$ with
$$
\tfrac 1q=\tfrac1{2q_0}+\tfrac1{2q_1}, \qquad q_i = \tfrac{d+2}d p_i',
$$
such that $\scriptE$ extends as a bounded linear operator from $L^{p_i}$ to $L^{q_i}$, $i=0,1$.  By H\"older's inequality, boundedness of $\scriptE$, and the support conditions and pointwise boundedness of $h_n^j, h_n^{j'}$,
\begin{align*}
\|\scriptE h_n^j \scriptE h_n^{j'}\|_{\frac q2} &\leq \|\scriptE h_n^j\|_{q_0}\|\scriptE h_n^{j'}\|_{q_1} \lesssim \|h_n^j\|_{p_0} \|h_n^{j'}\|_{p_1} \\
&\lesssim 2^{j+j'} (\rho_n^j)^{\frac d{p_0}-\frac d{p}} (\rho_n^{j'})^{\frac d{p_1}-\frac dp} = 2^{j+j'} \bigl(\tfrac{\rho_n^{j'}}{\rho_n^j}\bigr)^{\frac d{p_1}-\frac dp} \to 0.
\end{align*}
By symmetry in $j,j'$, it remains to consider the case when $\lim\tfrac{\rho_n^j}{\rho_n^{j'}} \in (0,\infty)$ and $\lim_{n \to \infty}(\rho_n^j)^{-1}\dist([\nu_n^j],[\nu_n^{j'}]) = \infty$.  Of course, this implies that $\rho_n^j,\rho_n^{j'} \to 0$.  Set $r_n = \tfrac1{100} \dist([\nu_n^j],[\nu_n^{j'}])$.  Then $\tfrac{\rho_n^j}{r_n} \to 0$, and hence, for large $n$, $\tau_n^j \subseteq B([\nu_n^j],r_n)$, $\tau_n^{j'} \subseteq B([\nu_n^{j'}],r_n)$.  By Lemma~\ref{L:bilin}, for some $s<p$, 
\begin{align*}
\|\scriptE h_n^j \scriptE h_n^{j'}\|_{\frac q2} &\lesssim r_n^{-(\frac{2d}{s} - \frac{2d}p)}\|h_n^j\|_s\|h_n^{j'}\|_s \leq \bigl(\tfrac{\rho_n^j}{r_n}\bigr)^{d(\frac1s-\frac1p)}\bigl(\tfrac{\rho_n^{j'}}{r_n}\bigr)^{d(\frac1s-\frac1p)} \to 0.
\end{align*}
We have thus established \eqref{E:Lq bilin decoup}, completing the proof of (iv) in Theorem~\ref{T:freq}.  
\end{proof}

\section{Scale one spatial decomposition:  Proof of Theorem~\ref{T:space1}} \label{S:scale 1 profiles}

In this section, we establish a finer decomposition for sequences that are (almost) pointwise bounded, in the sense that 
$$
\lim_{M \to \infty} \limsup_{n \to \infty} \|\scriptE f_n\chi_{\{|f_n|>M\}}\|_q = 0.
$$

\begin{lemma}\label{L:inv rest scale 1}
If, under the hypotheses of Theorem~\ref{T:space1}, we have in addition that
$$
\limsup_{n \to \infty} \|f_n\|_p \leq A < \infty \qtq{and} \liminf_{n \to \infty} \|\scriptE f_n\|_q \geq B > 0,
$$
then there exists $\{x_n\} \subseteq \R^{d+1}$ such that after passing to a subsequence,
\begin{equation} \label{E:scale 1 x_n}
e^{-ix_n\omega}f_n \rightharpoonup \phi, \qtq{weakly in} L^p, \qtq{with} \|\scriptE \phi\|_q \gtrsim B(\tfrac BA)^C.
\end{equation}
Furthermore, along this subsequence,
\begin{equation} \label{E:scale 1 Lq decoup}
\lim_{n \to \infty} \|\scriptE f_n\|_q^q - \|\scriptE(f_n-e^{ix_n\omega}\phi)\|_q^q - \|\scriptE \phi\|_q^q = 0.
\end{equation}
\end{lemma}

The proof is simplest if we divide into two cases.  

\begin{proof}[Proof when $q>\tfrac{d+2}d p'$]
It follows from our hypotheses that the $\scriptE f_n$ are bounded, continuous functions.  After a modulation and multiplication by unimodular constants, we may assume that $\scriptE f_n(0) = \|\scriptE f_n\|_\infty$.  Under this normalization, we will prove \eqref{E:scale 1 x_n} with $x_n \equiv 0$.  After passing to a subsequence, there exists $\phi \in L^p$ such that $f_n \rightharpoonup \phi$, weakly in $L^p$.  Therefore
$$
\|\scriptE \phi\|_\infty \geq |\scriptE \phi(0)| = \lim|\scriptE f_n(0)| = \lim\|\scriptE f_n\|_\infty.
$$
By hypothesis, there exists $s<q$ such that $\scriptE$ extends as a bounded linear operator from $L^p$ to $L^s$.  By H\"older's inequality and our hypotheses,  
$$
B \leq \limsup\|\scriptE f_n\|_q \leq \limsup\|\scriptE f_n\|_s^{1-\frac sq}\|\scriptE f_n\|_\infty^{\frac sq} \lesssim A^{1-\frac sq}\lim \|\scriptE f_n\|_\infty^{\frac sq}.
$$
Therefore $\|\scriptE\phi\|_\infty \gtrsim B(\tfrac BA)^{\frac qs-1}$.  On the other hand, $\scriptE \phi = (\scriptE \phi)*h$ whenever $\hat h \equiv 1$ on $\S^d$, so by Young's inequality, $\|\scriptE f_n\|_\infty \lesssim \|\scriptE f_n\|_q$, and \eqref{E:scale 1 x_n} follows.  

The final conclusion, \eqref{E:scale 1 Lq decoup}, follows from the Brezis--Lieb lemma, since $\scriptE e^{-ix_n\omega}f_n \to \scriptE \phi$ pointwise, by virtue of the weak convergence.  \end{proof}

\begin{proof}[Proof when $q = \tfrac{d+2}d p'$]
In this argument, we will use the fact that a cap is the intersection of a union of antipodal cubes with $S^d$.   For $\tau \in \scriptD$, we denote by $Q_\tau$ a union of two rather smaller (relative to the cubes) antipodal parallelepipeds whose intersection with $S^d$ also gives $\tau$, but whose volume is comparable to that of the convex hull of $\tau$.  More precisely, due to the curvature of the sphere, $|Q_\tau| \sim |\tau|^{\frac q{p'}}$.  

Now let $0 < \eps < \tfrac B2$ be sufficiently small for later purposes.  By hypothesis, there exists $M>0$ (over which we have no control) such that
$$
\limsup_{n \to \infty} \|\scriptE (f_n \chi_{\{|f_n|>M\}})\|_q < \eps.
$$
 
By \eqref{E:scalable} and our hypotheses, 
$$
B(\tfrac B A )^{\frac{1-\theta}\theta} \lesssim \liminf_{n \to \infty} \sup_{\tau_n \in \scriptD} |\tau_n|^{-\frac1{p'}}\|\scriptE(f_n \chi_{\{|f_n| \leq M\}})_{\tau_n}\|_\infty.
$$
By H\"older's inequality, for each $n$, the above supremum may be taken over caps whose volumes are bounded below by a constant depending on $M$.  As there are only a finite number of such caps, after passing to a subsequence, there is a single $\tau_M$ that realizes the supremum for all sufficiently large $n$.  In other words,
\begin{equation} \label{E:lim inf Ef_n<M}
B(\tfrac B A)^{\frac{1-\theta}\theta} \lesssim \liminf_{n \to \infty} |\tau_M|^{-\frac1{p'}}\|\scriptE(f_n\chi_{\{|f_n|<M\}})_{\tau_M}\|_\infty.
\end{equation}
We may assume, after modulation, that $\scriptE(f_n\chi_{\{|f_n|<M\}})_{\tau_M}(0) = \|\scriptE(f_n\chi_{\{|f_n|<M\}})_{\tau_M}\|_\infty$.  

Passing to a subsequence, the weak limits
$$
\phi^g:=\wklim f_n\chi_{\{|f_n| \leq M\}}, \qquad \phi^b:=\wklim f_n\chi_{\{|f_n|>M\}}
$$
exist in $L^p$.  In particular, 
$$
\lim \scriptE(f_n\chi_{\{|f_n| \leq M\}})_{\tau_M}(x) = \scriptE (\phi^g)_{\tau_M}(x), \qquad \lim \scriptE(f_n\chi_{\{|f_n|>M\}})(x) = \scriptE \phi^b(x),
$$ 
for all $x \in \R^{1+d}$.  By Fatou,
$$
\|\scriptE \phi^b\|_q \leq \liminf_{n \to \infty} \|\scriptE(f_n\chi_{\{|f_n|>M\}})\|_q < \eps.
$$
On the other hand, by Young's inequality and $q' > 1$,
$$
|{\tau_M}|^{-\frac1{p'}} \|\scriptE \phi^b_{\tau_M}\|_\infty \leq |{\tau_M}|^{-\frac1{p'}} \|\widehat{\chi_{Q_{\tau_M}}}\|_{q'} \|\scriptE \phi^b\|_q \lesssim \|\scriptE \phi^b\|_q.
$$
Hence by the triangle inequality, \eqref{E:lim inf Ef_n<M}, Young's inequality, and $q'>1$,
$$
B(\tfrac B A)^{\frac{1-\theta}\theta} \lesssim |{\tau_M}|^{-\frac1{p'}} \|\scriptE \phi^g_{\tau_M}\|_\infty \leq |{\tau_M}|^{-\frac1{p'}} \|\widehat{\chi_{Q_{\tau_M}}}\|_{q'} \|\scriptE \phi^g\|_q \lesssim \|\scriptE \phi^g\|_q.
$$
Setting $\phi:=\phi^g+\phi^b = \wklim f_n$, the lower bound on the extension of $\phi^g$ and upper bound on the extension of $\phi^b$ give
$$
B({ \tfrac B A})^{\frac{1-\theta}\theta} \lesssim \|\scriptE \phi\|_q,  
$$
provided $\eps \ll B({\tfrac B A})^{\frac{1-\theta}\theta}$.  

The final inequality, \eqref{E:scale 1 Lq decoup}, follows as above from the Brezis--Lieb lemma and pointwise convergence $\scriptE e^{-ix_n\omega}f_n \to \scriptE \phi$. 
\end{proof}

In the special case $p=2$, iteratively applying Lemma~\ref{L:inv rest scale 1} and using elementary Hilbert space identities yields the following $L^2$ profile decomposition.  

\begin{lemma} \label{L:L2 profile scale 1}
If $q \geq \tfrac{2(d+2)}d$ and $\{f_n\}$ is a bounded sequence in $L^2(\S^d)$ satisfying
\begin{equation} \label{E:scale 1 condition}
\lim_{M \to \infty} \limsup_{n \to \infty} \|\scriptE f_n \chi_{\{|f_n| > M\}}\|_q = 0,
\end{equation}
then, after passing to a subsequence, there exist $\{x_n^j\}_{j,n \in \N} \subseteq \R^{d+1}$ obeying
\begin{equation} \label{E:L2 profile scale 1 asymp orthog}
\lim_{n \to \infty} |x_n^j-x_n^{j'}| = \infty, \qtq{for} j \neq j',
\end{equation}
and weak limits $\phi^j = \wklim_{n \to \infty} e^{-ix_n^j\omega}f_n \in L^2$, such that for every $J \in \N$,
\begin{gather}\label{E:L2 profile scale 1 L2 orthog}
\lim_{n \to \infty} \|f_n\|_2^2 - \sum_{j=1}^J\|\phi^j\|_2^2 - \|r_n^J\|_2^2 = 0\\\label{E:L2 profile scale 1 Lq orthog}
\lim_{n \to \infty} \|\scriptE f_n\|_q^q - \sum_{j=1}^J \|\scriptE \phi^j\|_q^q - \|\scriptE r_n^J\|_q^q = 0,
\end{gather}
and, moreover,
\begin{equation} \label{E:L2 profile scale 1 small err}
\lim_{J \to \infty} \lim_{n \to \infty} \|\scriptE r_n^J\|_q^q = 0.  
\end{equation}
\end{lemma}

We briefly sketch the proof of the lemma, which follows a well-known outline (see, for instance, \cite{ClayNotes}).  Only minor modifications to the familiar argument are needed to use the condition \eqref{E:scale 1 condition}.  

\begin{proof}[Sketch proof of Lemma~\ref{L:L2 profile scale 1}]
We set $r_n^0:=f_n$.  Given some $r_n^J$, we are done (setting $\phi^{J'}=0$ for $J'>J$) if $\lim_{n\to \infty} \|\scriptE r_n^J\|_q = 0$.  Otherwise, we apply Lemma~\ref{L:inv rest scale 1}, which produces a nonzero weak limit $\phi^{J+1}$ and sequence $\{x_n^{J+1}\} \subseteq \R^{d+1}$; we then set $r_n^{J+1} := r_n^J-e^{ix_n^{J+1}\omega}\phi^{J+1}$.  If \eqref{E:L2 profile scale 1 asymp orthog} failed for some $j<j'$ but held with $j'$ replaced by any $j<i<j'$, then $x_n^j-x_n^{j'}$ would converge along a subsequence, and, after a bit of algebraic manipulation, we can deduce that $\phi^{j'} = 0$.  Thus \eqref{E:L2 profile scale 1 asymp orthog} holds.  The hypothesis \eqref{E:scale 1 condition} continues to hold with $r_n^{J+1}$ in place of $f_n$, as a finite number of $L^p$ functions have been subtracted from the $f_n$.  Equation \eqref{E:L2 profile scale 1 L2 orthog} follows from basic Hilbert space manipulations, and \eqref{E:L2 profile scale 1 Lq orthog} may be proved inductively using Br\'ezis--Lieb.  Finally, \eqref{E:L2 profile scale 1 small err} holds because \eqref{E:scale 1 x_n} and \eqref{E:L2 profile scale 1 L2 orthog} give a lower bound for $\|r_n^J\|_2^2-\|r_n^{J+1}\|_2^2$ whenever $\lim_{n \to \infty} \|\scriptE r_n^J\|_q^q \not\to 0$.  
\end{proof}

For $p \neq 2$, $L^p$ is not a Hilbert space, and the direct analogue of \eqref{E:L2 profile scale 1 L2 orthog} may fail.  Instead, we will prove the $L^p$ almost orthogonality estimates \eqref{E:Lp orthog 1} by defining and bounding a family of vector-valued operators.  Fix a nonnegative, smooth, radial function $\psi$ on $\R^{d+1}$ with compact support in the unit ball and $\int_{\R^d} \psi(\xi',0)\, d\xi' = 1$.  (Since $\psi$ is radial, it thus has integral 1 on every hyperplane through the origin.)  For $0 < r \leq 1$, define $\psi_r(\zeta):=r^{-d}\psi(r^{-1}\zeta)$.  We define
\begin{gather*}
(\pi_r)_n^j f(\omega):=\int \psi_r(\omega-\nu)e^{-ix_n^j\nu}f(\nu)\,d\sigma(\nu),\qquad (\Pi_r)_n^J f:=((\pi_r)_n^j f)_{j=1}^J.
\end{gather*}
We recall $\tilde p:=\max\{p,p'\}$.  

\begin{lemma} \label{L:Lp almost orthog 1}
Let $1 < p \leq \infty$.  Assume that the sequences $\{x_n^j\}_{j,n \in \N}$ obey $\lim_{n \to \infty}|x_n^j-x_n^{j'}| = \infty$ for all $j \neq j'$.  Then the $(\Pi_r)_n^J$ map $L^p$ boundedly into $\ell^{\tilde p}(L^p)$, with operator norms bounded uniformly in $r,n$.  Moreover, 
\begin{equation} \label{E:op norm 1 1}
\lim_{r \to \infty} \lim_{n \to \infty} \|(\Pi_r)_n^J\|_{L^p \to \ell^{\tilde p}(L^p)} = 1.
\end{equation}
Finally, given sequences of functions $f_n = \sum_{j=1}^J e^{ix_n^j \omega}\phi^j + r_n^J$, with $\{f_n\}$ bounded in $L^p$,
satisfying $\phi^j = \wklim e^{-ix_n^j \omega}f_n$, for each $j \in \N$, we have
\begin{gather} \label{E:pinj f phi 1}
\lim_{r \to 0} \lim_{n \to \infty} \|(\pi_r)_n^j f_n - \phi^j\|_{L^p} = 0\\ \label{E:pinj* phi 1}
\lim_{r \to 0} \lim_{n \to \infty} \|[(\pi_r)_n^j]^*\phi^j - e^{ix_n^j\omega}\phi^j\|_{L^p} = 0.
\end{gather}
\end{lemma}

\begin{proof}[Proof of Lemma~\ref{L:Lp almost orthog 1}]
We will be brief.  Verification that the $\psi_r$ approximate the identity is routine; boundedness of the $(\Pi_r)_n^J$ and the limit \eqref{E:pinj* phi 1} immediately follow.  By our weak limit hypothesis on the $\phi^j$ and the dominated convergence theorem (with constant dominating function), for each $r$,
$$
(\pi_r)_n^j f_n(\omega) \to \int \psi_r(\omega-\nu)\phi^j(\nu)\, d\sigma(\nu), \qtq{in $L^p$, as $n \to \infty$.}
$$
Equation \eqref{E:pinj f phi 1} follows immediately.

We will verify the dual form of \eqref{E:op norm 1 1}, namely, that 
$$
\lim_{r \to 0} \lim_{n \to \infty} \|[(\Pi_r)_n^J]^*\|_{\ell^{\tilde p'}(L^p) \to L^p} \to 1, \qquad 1 \leq p \leq \infty.
$$
For the convenience of the reader, we record
$$
[(\Pi_r)_n^J]^*(\phi^j)(\nu) = \sum_{j=1}^J e^{ix_n^j \nu}\int \psi_r(\nu-\omega)\phi^j(\omega)\, d\sigma(\omega).
$$
That the limit of the operator norms is bounded below by 1 is elementary, as can be seen by applying the $[(\Pi_r)_n^J]^*$ to the vector-valued constant function $(1,0,\ldots,0)$.  In the cases $p=1,\infty$, the upper bound is a direct consequence of the triangle inequality, H\"older's inequality, and 
$$
\int \psi_r(\omega-\nu)\,d\sigma(\nu) = \int \psi_r(e_1-\nu)\,d\sigma(\nu) \to 1, \qtq{for any} \omega \in \S^d.
$$
By interpolation, it remains to verify the $p=2$ case, for which it suffices to prove that $\|(\Pi_r)_n^J[(\Pi_r)_n^J]^*\|_{\ell^2(L^2) \to \ell^2(L^2)} \to 1$.  We expand
\begin{align*}
\|(\Pi_r)_n^J[(\Pi_r)_n^J]^*(\phi^j)_{j \leq J}\|_{\ell^2(L^2)}^2 
=\sum_j \|\sum_{k=1}^J \int\phi^k(\vartheta) (K_r)^{jk}_n(\vartheta,\omega)\,d\sigma(\vartheta)\|_{L^2_\omega}^2,
\end{align*}
where
$$
(K_r)^{jk}_n(\vartheta,\omega) = \int \psi_r(\omega-\nu)\psi_r(\nu-\vartheta)e^{i(-x_n^j+x_n^k)\nu}d\sigma(\nu).
$$
Let 
$$
(A_r)^{jk}_n := \|(K_r)^{jk}_n\|_{L^\infty_\vartheta L^1_\omega} \|(K_r)^{jk}_n\|_{L^\infty_\omega L^1_\vartheta}.
$$
When $j \neq k$ and $r > 0$, $(K_r)^{jk}_n \to 0$ uniformly as $n \to \infty$ by stationary phase and $|x_n^j-x_n^k| \to \infty$, so the off-diagonal terms satisfy $(A_r)^{jk}_n \to 0$ as $n \to \infty$.  By construction of $\psi_r$, $(A_r)^{jj}_n$ (which is independent of $n$) tends to 1 as $r \to 0$.  By the elementary inequality
$$
|\sum_{j=1}^J x_j|^2 \leq (1+\eps)x_1^2 + C_{\eps,J}\sum_{j=2}^J x_j^2
$$
and Schur's test,
$$
\|\sum_{k=1}^J \int\phi^k(\vartheta) (K_r)_n^{jk}(\vartheta,\omega)\,d\sigma(\vartheta)\|_{L^2_\omega}^2
\leq (1+\eps)(A_r)_n^{jj}\|\phi^j\|_2^2 + C_{\eps,J} \sum_{k \neq j} (A_r)_n^{jk}\|\phi^k\|_2^2,
$$
and \eqref{E:op norm 1 1} follows.  
\end{proof}

\begin{lemma} \label{L:decoup 1}
Under the hypotheses of Theorem~\ref{T:space1}, suppose that we are given sequences $\{x_n^j\}_{j,n \in \N}$ with $|x_n^j-x_n^{j'}| \to \infty$ for $j \neq j'$ and $\{f_n\}$ such that the weak limits 
$$
\phi^j:=\wklim e^{-ix_n^j\omega}f_n
$$
exist.  Define
$$
r_n^J:=f_n-\sum_{j=1}^J e^{ix_n^j\omega}f_n.
$$
Then \eqref{E:Lp orthog 1} and \eqref{E:Lq orthog 1} both hold.  
\end{lemma}

\begin{proof}[Proof of Lemma~\ref{L:decoup 1}]
The inequalities in \eqref{E:Lp orthog 1} follow directly from Lemma~\ref{L:Lp almost orthog 1} and the definition of the $r_n^J$.  We will use the Brezis--Lieb lemma to prove \eqref{E:Lq orthog 1}.  Set $r_n^0:=f_n$.  By hypothesis, for $j \neq j'$, $\wklim e^{i(x_n^j-x_n^{j'})}\phi^j = 0$.  Therefore $\phi^j = \wklim e^{-ix_n^j}r_n^{j-1}$, for $j\geq 1$.  Applying the extension, $\scriptE(e^{-ix_n^j\omega}r_n^{j-1}-\phi^j) \to 0$, pointwise, and so by the Brezis--Lieb lemma, 
$$
\lim_{n \to \infty} \|\scriptE(e^{-ix_n^j\omega}r_n^{j-1})\|_q^q - \|\scriptE \phi^j\|_q^q - \|\scriptE (e^{-ix_n^j\omega}r_n^{j-1}-\phi^j)\|_q^q = 0.  
$$
Summing the preceding identity over $j=1,\ldots,J$ and using $r_n^j = r_n^{j-1}-e^{ix_n^j\omega}\phi^j$ establishes \eqref{E:Lq orthog 1}.  
\end{proof}

With the above lemmas in place, we are now ready to complete the proof of Theorem~\ref{T:space1}.  

\begin{proof}[Proof of Theorem~\ref{T:space1}]
We may assume that $\limsup \|f_n\|_p = 1$.  When $p=2$, the conclusions of Lemma~\ref{L:L2 profile scale 1} are stronger than we need, so it suffices to consider pairs $(p,q)$ meeting the hypotheses of our theorem in the case $p \neq 2$.  In light of Lemma~\ref{L:decoup 1}, it suffices to prove that there exist $\{x_n^j\}_{j,n \in \N}$ obeying \eqref{E:xnj-xnj' scale 1} such that the resulting remainder terms $r_n^J$ have small extension, i.e.,\ that  \eqref{E:ErnJ 0 1} holds.  

Given $M \in \N$, we set $f_n^M:=f_n\chi_{\{|f_n| \leq M\}}$.  Let $\eps>0$ and take $M_\eps$ sufficiently large that 
$$
\limsup_{n \to \infty} \|\scriptE (f_n-f_n^M)\|_q < \eps,  
$$
when $M\geq M_\eps$.  

The advantage of working with the truncation $\{f_n^M\}$ is that it forms a bounded sequence in every Lebesgue space (albeit with a bad, $M$-dependent, bound), putting us in a position to apply  Lemma~\ref{L:L2 profile scale 1}.  To this end, set $q_1:=\tfrac{2(d+2)}d$ and choose an exponent pair $(p_0,q_0)$ meeting the hypotheses on $(p,q)$ from Theorem~\ref{T:space1}, as well as the condition
$$
(\tfrac1p,\tfrac1q) = (1-\theta)(\tfrac1{p_0},\tfrac1{q_0}) + \theta(\tfrac12,\tfrac1{p_1}),
$$
for some $0 < \theta < 1$.  

By Lemma~\ref{L:L2 profile scale 1}, after passage to subsequence (independent of $M$ by standard diagonalization arguments), there exist points $\{x_n^{M,j}\}_{j,n \in \N}$ and weak limits $\phi^{M,j}$ such that \eqref{E:L2 profile scale 1 L2 orthog}, \eqref{E:L2 profile scale 1 Lq orthog}, and \eqref{E:L2 profile scale 1 small err} all hold, with the superscript $M$ inserted where appropriate.  Since $e^{-ix_n^{M,j}\omega}f_n^M \rightharpoonup \phi^{M,j}$ weakly in both $L^p$ and in $L^{p_0}$, we may also apply Lemma~\ref{L:decoup 1} with exponents $(p,q)$ and $(p_0,q_0)$.  By \eqref{E:Lq orthog 1}, 
$$
\limsup_{n \to \infty} \|\scriptE r_n^{M,J}\|_{q_0} \leq \limsup_{n \to \infty}\|\scriptE f_n^M\|_{q_0} \lesssim \limsup_{n \to \infty}\|f_n^M\|_{p_0} \lesssim_M 1,
$$ 
for all $M,J$.  Therefore, by H\"older's inequality and \eqref{E:L2 profile scale 1 small err}, 
\begin{equation} \label{E:rnJM 0 1}
\lim_{J \to \infty} \limsup_{n \to \infty} \|\scriptE r_n^{M,J}\|_q = 0, \qtq{for all $M$.}  
\end{equation}

To conclude, we need to remove the dependence on $M$ in \eqref{E:rnJM 0 1}.  We begin by showing that non-negligible profiles $\scriptE e^{ix_n^{j,M}\omega}\phi^{j,M}$ cannot wander around too much as $M$ varies.  

After passing to a subsequence, we may assume that for any $M,j$ and $M',j'$, either $\lim_{n \to \infty}|x_n^{M,j}-x_n^{M',j'}| =\infty$ or $x_n^{M,j}-x_n^{M',j'}$ converges in $\R^{d+1}$, as $n \to \infty$.  In fact, in the latter case, we may assume that $x_n^{M,j} \equiv x_n^{M',j'}$, simply by modulating our $\phi^{M,j}$ as needed.   

Now, let $C$ be a sufficiently large constant, and suppose that (after reordering the $\phi^{j,M}$ and perhaps inserting some zero profiles) we had distinct sequences $\{x_n^j\}$, $1 \leq j \leq J_\eps:=C\eps^{-q}$ such that $x_n^{M,j} \equiv x_n^j$ for some $M \geq M_\eps$ with $\|\scriptE \phi^{M,j}\|_q > 2\eps$.  By construction, $|x_n^j-x_n^{j'}| \to \infty$ whenever $j \neq j'$.  Passing to a subsequence, we have weak limits $e^{-ix_n^j\omega}f_n \rightharpoonup \phi^j$, weakly in $L^p$, for each $1 \leq j \leq J_\eps$.  By Fatou, 
$$
\|\scriptE \phi^j-\scriptE \phi^{M,j}\|_q \leq \limsup_{n \to \infty} \|\scriptE f_n-\scriptE f_n^M\|_q < \eps.
$$
Therefore $\|\scriptE \phi^j\|_q > \eps$, $1 \leq j \leq J_\eps$.  By Lemma~\ref{L:decoup 1} (namely, inequality \eqref{E:Lq orthog 1}), 
$$
J_\eps \eps^q < \sum_{j=1}^{J_\eps}\|\scriptE \phi^j\|_q^q \leq \lim_{n \to \infty} \|\scriptE f_n\|_q^q \lesssim 1,
$$
a contradiction.  Thus, after reordering, $x_n^{M,j} \equiv x_n^j$ whenever $M \geq M_\eps$, $j \leq J_\eps$ and 
\begin{equation} \label{E:phiMj small 1}
\|\scriptE \phi^{M,j}\|_q < \eps,
\qtq{whenever $M \geq M_\eps$ and $j > J_\eps$.}
\end{equation}

Inequality \eqref{E:phiMj small 1} will give us uniform control on the extensions of the $r_n^{M,J_\eps}$.  Recalling that $\tilde p = \max\{p,p'\}<q$,  choose, for each $M \geq M_\eps$, some $J_{\eps,M}$ sufficiently large that $\|\scriptE r_n^{M,J_{\eps,M}}\|_q < \eps^{q-\tilde p}$.  We will show that $\|\scriptE(r_n^{M,J_\eps}-r_n^{M,J_{\eps,M}})\|_q \lesssim \eps^{q-\tilde p}$.  Noting that $$
r_n^{M,J_\eps}-r_n^{M,J_{\eps,M}} = \sum_{J_\eps<j<J_{\eps,M}} e^{ix_n^{M,j}\omega}\phi^{M,j},
$$
we apply Lemma~\ref{L:decoup 1} and H\"older's inequality to obtain
\begin{align*}
    &\limsup_{n \to \infty} \, \bigl\|\sum_{J_\eps<j<J_{\eps,M}} \scriptE e^{ix_n^{M,j}\omega}\phi^{M,j}\bigr\|_q 
    \leq 
    \sum_{J_\eps<j<J_{\eps,M}} \|\scriptE \phi^{M,j}\|_q^q 
    \\&\qquad\leq \eps^{q-\tilde p} \sum_{J_\eps<j<J_{\eps,M}} \|\scriptE \phi^{M,j}\|_q^{\tilde p}
    \leq \eps^{q-\tilde p} A_{p \to q}^{\tilde p} \sum_{J_\eps<j<J_{\eps,M}} \| \phi^{M,j}\|_p^{\tilde p}
    \\&\qquad\leq \eps^{q-\tilde p} A_{p \to q}^{\tilde p} \limsup\|f_n\|_p^{\tilde p} \lesssim \eps^{q-\tilde p}.
\end{align*}

It remains to transfer the bound $\|\scriptE r_n^{M,J_\eps}\|_q \lesssim \eps^{q-\tilde p}$ to $\scriptE r_n^{J_\eps}$.  Let $1 \leq j \leq J_\eps$.  By Fatou and our  assumption, 
$$
\lim_{M \to \infty} \|\scriptE (\phi^j-\phi^{M,j})\|_q \leq \lim_{M \to \infty} \limsup_{n \to \infty} \|\scriptE(f_n-f_n^M)\|_q = 0.
$$
Hence by the triangle inequality, $$
\lim_{M \to \infty} \|\scriptE(r_n^{J_\eps}-r_n^{M,J_\eps})\|_q = 0,
$$
and so we have the desired inequality $\|\scriptE r_n^{J_\eps}\|_q < \eps$, completing the proof of Theorem~\ref{T:space1}.
\end{proof}

\section{Large scale spatial decomposition:  Proof of Theorem~\ref{T:large space}} \label{S:concentrating profiles}

We begin by recording the connection between the spherical and parabolic extension operators at small frequency scales.  

\begin{lemma} \label{L:sphere to parab}
Let $1<p<\tfrac{2(d+1)}d$ and $q=\tfrac{d+2}dp'$.  Assume that the restriction conjecture for $\scriptE$ holds on a neighborhood of $(p,q)$. Let $\lambda_n \searrow 0$ and $\phi \in L^p(\R^d)$.  Define
\begin{equation} \label{E:def gn}
g_n (\omega):=   \lambda_n^{-d/p}\phi(\lambda_n^{-1}\omega')\chi_{\{\omega_1>0\}}\chi_{\{|\omega'|<\tfrac12\}}.
\end{equation}
Then 
\begin{equation} \label{E:sphere to parab}
\lim_{n\to \infty} \| \scriptE g_n - \lambda_n^{\frac{d+2}q}e^{ ix_1}\scriptE_\P \phi(-\lambda_n^2x_1,\lambda_n x') \|_{L^q} = 0.
\end{equation}
\end{lemma}

Lemma~\ref{L:sphere to parab} is proved in the case $p=2$ in \cite{FLS}; we make the simple adaptation here to the case of general $p$ (for which we make no \textit{a priori} assumption of boundedness of $\scriptE_\P$) for the convenience of the reader.

\begin{proof}[Proof of Lemma~\ref{L:sphere to parab}]
It suffices to prove that 
$$
\lim_{n\to \infty} \| \lambda_n^{-\frac{d+2}q} e^{i\lambda_n^{-2}x_1} \scriptE g_n(-\lambda_n^{-2}x_1,\lambda_n^{-1}x') - \scriptE_\P \phi \|_{L^q} = 0,
$$
and we assume initially that $\phi \in C^\infty_{cpct}$; therefore, the $L^q$ norms in the above limit are finite for each $n$ by stationary phase.  Set
$$
G_n(x):= \lambda_n^{-\frac{d+2}q} e^{i\lambda_n^{-2}x_1} \scriptE g_n(-\lambda_n^{-2}x_1,\lambda_n^{-1}x'). 
$$
After a change of variables, we see that for sufficiently large $n$,
$$
G_n(x) = \int e^{i(-x_1,x')(\lambda_n^{-2}[\sqrt{1-|\lambda_n\xi|^2}-1],\xi)}\phi(\xi)\, d\xi.
$$
Examining the phase function, $G_n \to \scriptE_\P \phi$, pointwise.  Moreover, by stationary phase, $|G_n(x)| \lesssim_\phi \jp{x}^{-\frac d2}$.  Therefore by dominated convergence, $G_n \to \scriptE \phi$ in $L^q$.  

From our assumption, having proved the lemma in the case of $C^\infty_{cpct}$ functions implies, in addition, that $\|\scriptE_\P \phi\|_q \lesssim \|\phi\|_p$ for $\phi \in C^\infty_{cpct}$.  Therefore $\scriptE_\P$ extends as a bounded linear operator from $L^p$ to $L^q$, and we may conclude that the lemma also holds for general $L^p$ functions by standard approximation arguments.  
\end{proof}

Next, we isolate a nonzero weak limit in bounded, concentrating sequences with nonnegligible extensions.  

\begin{lemma}\label{L:inv rest large scale}
Let $1<p<\tfrac{2(d+1)}d$ and $q=\tfrac{d+2}dp'$.  Assume that the restriction conjecture for $\scriptE$ holds on a neighborhood of $(p,q)$.  Let $\lambda_n \searrow 0$ and assume that
\begin{equation} \label{E:Efn>M large space +}
\lim_{M \to \infty} \limsup_{n \to \infty} \|\scriptE f_n\chi_{\{|f_n|>M\lambda_n^{-d/p}\} \cup \{|\omega-e_1|>M\lambda_n\}}\|_q = 0,
\end{equation}
for some sequence $\{f_n\} \subseteq L^p(\S^d)$, with $\limsup\|f_n\|_p \leq A$ and $\liminf\|\scriptE f_n\|_q \geq B > 0$.  After passing to a subsequence, there exists $\{x_n\} \subseteq \R^{d+1}$ such that
$$
\lambda_n^{d/p} e^{-ix_n(\sqrt{1-|\lambda_n\xi|^2},\lambda_n \xi)}f_n(\sqrt{1-|\lambda_n\xi|^2},\lambda_n\xi) \chi_{\{|\xi| < \frac12 \lambda_n^{-1}\}} \rightharpoonup \phi,
$$
weakly in $L^p(\R^d)$, with $\|\scriptE_\P \phi\|_q \gtrsim B(\tfrac B A)^C$.    
\end{lemma}

\begin{proof}[Proof of Lemma~\ref{L:inv rest large scale}]
Given $M \in \N$, we set
$$
f_n^{>M}:= f_n\chi_{\{|f_n|>M\lambda_n^{-d/p}\} \cup \{|\omega-e_1|>M\lambda_n\}}, \qtq{and} f_n^M:=f_n-f_n^{>M}.
$$
Let $\eps>0$ sufficiently small for later purposes.  By hypothesis, there exists $M:=M_\eps\in \N$ such that, after passing to a subsequence,
\begin{equation} \label{E:Efn>M small}
\|\scriptE f_n^{>M}\|_q <\eps, \qtq{for all $n$.}
\end{equation}
As long as $\eps<\tfrac B2$, after passing to a further subsequence, 
$$
\|f_n^M\|_p \leq A, \qtq{and} \|\scriptE f_n^M\|_q \geq \tfrac B2, \qtq{for all $n$.}
$$
By \eqref{E:scalable}, there exists a sequence $\{\tau_n\} \subseteq \scriptD$ such that for all $n$,
$$
B(\tfrac BA)^{\frac{1-\theta}\theta} \lesssim |\tau_n|^{-\frac1{p'}}\|\scriptE (f_n^M)_{\tau_n}\|_\infty.
$$

We will show that after rescaling by $\lambda_n$, the $\tau_n$ have a convergent subsequence. On the one hand, by \eqref{E:Efn>M small}, for $\eps \ll B(\tfrac BA)^{\frac{1-\theta}\theta}$, each $\tau_n$ must intersect $\{|\xi-e_1|<M\lambda_n\}$.  On the other hand, by H\"older's inequality,
\begin{align*}
|\tau_n|^{-\frac1{p'}}\|\scriptE(f_n^M)_{\tau_n}\|_\infty &\leq |\tau_n|^{-\frac1{p'}}\|f_n^M\|_\infty|\supp f_n^M| \\
&\lesssim_M \min\{|\tau_n|^{-1/p'}\lambda_n^{d/p'},|\tau_n|^{1/p}\lambda_n^{-d/p}\}
\end{align*}
where $|\supp f_n^M|$ denotes the measure of the support of $f_n^M.$
Therefore $1 \lesssim_{A,B,M} \min\{|\tau_n|^{-1/p'}\lambda_n^{d/p'},|\tau_n|^{1/p}\lambda_n^{-d/p}\}$, which implies  
 $|\tau_n| \sim_{A,B,M} \lambda_n^d$.  Therefore, after passing to a subsequence, $\chi_{\tau_n}(\sqrt{1-|\lambda_n\xi|^2},\lambda_n\xi) \to \chi_\tau(\xi)$, for some cube $\tau \subseteq \R^d$, pointwise, a.e.  

We may assume, after modulation and multiplication by a constant that 
$$
\scriptE(f_n^M)_{\tau_n}(0) = \|\scriptE (f_n^M)_{\tau_n}\|_\infty.
$$
Passing to a subsequence, the weak limits
\begin{gather*}
    \phi^g(\xi):= \wklim \lambda_n^{d/p}f_n^M(\sqrt{1-|\lambda_n\xi|^2},\lambda_n\xi)\\
        \phi(\xi):=\wklim \lambda_n^{d/p}f_n(\sqrt{1-|\lambda_n\xi|^2},\lambda_n\xi)\chi_{\{|\xi|<\frac12\lambda_n^{-1}\}}
\end{gather*}
exist; 
we set $\phi^b:=\phi-\phi^g$.  By the dominated convergence theorem and the observation of the previous paragraph,
$$
(\phi^g(\xi))_\tau = \wklim \lambda_n^{d/p}(f_n^M)_{\tau_n}(\sqrt{1-|\lambda_n\xi|^2},\lambda_n\xi).  
$$

Standard convergence arguments give, 
$$
\scriptE_\P \phi(x) = \lim_{n \to \infty} \lambda_n^{-\frac{d+2}q}e^{-i\lambda_n^{-2}x_1}\scriptE(f_n)_{\tau_n}(\lambda_n^{-2}x_1,\lambda_n^{-1}x'),
$$
for all $x$, and analogous relations hold for $\phi^g$ and the $f_n^M$ (and hence for $\phi^b$ and the $f_n^{>M}$).  Therefore
$$
|\tau|^{-\frac1{p'}}\|\scriptE_\P(\phi^g)_\tau\|_\infty \gtrsim B(\tfrac BA)^{\frac{1-\theta}\theta}.
$$
By Young's convolution inequality and the observation that the measure of the convex hull of $\tau$ satisfies $|\ch \tau| \sim |\tau|^{\frac{d+2}d}$,
$$
\|\scriptE_\P(\phi^g)_\tau\|_q \gtrsim B(\tfrac BA)^{\frac{1-\theta}\theta}.
$$
On the other hand, by Fatou,
$$
\|\scriptE_\P(\phi^b)_\tau\|_q < \eps,  
$$
so by the triangle inequality,
$$
\|\scriptE_\P (\phi)_\tau\|_q \gtrsim B(\tfrac BA)^{\frac{1-\theta}\theta}.
$$
Hence by $L^q$ boundedness of Fourier multiplication by $\chi_{\tilde\tau}$, 
$$
\|\scriptE_\P \phi\|_q \gtrsim B(\tfrac BA)^{\frac{1-\theta}\theta}.
$$
\end{proof}

With Lemma~\ref{L:inv rest large scale} in place, we are ready for the $L^2$-based profile decomposition.  

\begin{lemma} \label{L:L2 profile large scale}
Theorem~\ref{T:large space} holds when $p=2$.  Moreover, with assumptions and notation as in the statement of that result, \\
\emph{(ii'-iii')} $\lim_{n \to \infty} \|f_n\|_2^2 - \sum_{j=1}^J(\|\phi^{j,+}\|_2^2 + \|\phi^{j,-}\|_2^2)-\|r_n^J\|_2^2 = 0$,
for all $J \in \N$.  
\end{lemma}

Essentially all of the key ingredients needed for this lemma were already established in \cite{FLS}; we provide details both for the convenience of the reader and because a key step, an improved Br\'ezis--Lieb Lemma, will be used in later arguments as well.   

\begin{proof}
We initially treat the two hemispheres separately, setting $f_n^\pm:=f_n\chi_{\{\pm\omega_1>0\}}$.  To decompose $f_n^+$, we set $r_n^{0,+}:=f_n^+$, and apply the following iterative process.  Given a bounded sequence of remainders $\{r_n^{J,+}\} \subseteq L^2(\S^d)$ obeying \eqref{E:Efn>M large space +}, we stop if $\lim\|\scriptE r_n^{J,+}\|_{q_2}=0$.  If this limit is nonzero, we apply Lemma~\ref{L:inv rest large scale}, obtaining a subsequence of $\{f_n\}$, points $\{x_n^{J+1,+}\} \subseteq \R^{d+1}$, and a weak limit $\phi^{J+1,+} \in L^2(\R^d)$. We set
$$
r_n^{J+1,+}:=r_n^{J,+} - e^{ix_n^{J+1,+}\omega}g_n^{J+1,+},
$$
with $g_n^{J+1,+}$ defined as in \eqref{E:def gn}.

That the $\{x_n^{j,+}\}$ move apart after parabolic rescaling follows from familiar arguments.  Namely, we suppose that there is some minimal superscript $j$ for which there exists some (minimal) superscript $j'>j$ with $\phi^{j,+},\phi^{j',+} \not\equiv 0$ and
$$
\lambda_n^2|(x_n^{j,+}-x_n^{j',+})_1|+\lambda_n|(x_n^{j,+}-x_n^{j',+})'| \not\to \infty.
$$
Passing to a subsequence,
$$
(\lambda_n^2(x_n^{j,+}-x_n^{j',+})_1,\lambda_n(x_n^{j,+}-x_n^{j',+})') \to y^{jj'}.
$$
Passing to a further subsequence,
$$
e^{i(x_n^{j,+}-x_n^{j',+})(\sqrt{1-|\lambda_n\xi|^2},\lambda_n\xi)} \to e^{i\theta^{jj'}}e^{iy^{jj'}(-\frac12|\xi|^2,\xi)},
$$
locally uniformly, for some $\theta^{jj'} \in [0,2\pi)$.  On the other hand,
$$
e^{i(x_n^k-x_n^{j'})(\sqrt{1-|\lambda_n\xi|^2},\lambda_n\xi)} \chi_{\{|\xi| \leq R\}} \rightharpoonup 0,
$$
weakly in $L^p$ for all $R$.  Noting that $r_n^{j'-1,+}=r_n^{j-1,+}-\sum_{k=j}^{j'-1} e^{ix_n^{k,+}\omega}g_n^{k,+}$,
\begin{align*}
    \phi^{j',+} &= \wklim \lambda_n^{\frac dp} e^{-ix_n^{j',+}(\sqrt{1-|\lambda_n\xi|^2},\lambda_n\xi)}g_n^{k,+}\\
    &=e^{i\theta^{jj'}}e^{iy^{jj'}(-\frac12|\xi|^2,\xi)}\phi^{j,+} - \sum_{k=j}^{j'-1}\wklim e^{i(x_n^{k,+}-x_n^{j',+})(\sqrt{1-|\lambda_n\xi|^2},\lambda_n\xi)}\phi^{k,+}\\
    &=0,
\end{align*}
a contradiction.  

Taking the complex conjugate and applying the preceding argument (along our new subsequence), we obtain decompositions for the lower hemisphere as well,
$$
f_n^- = \sum_{j=1}^J e^{ix_n^{j,-}\omega}g_n^{j,-} + r_n^{J,-}.
$$ 
Passing to a subsequence, for all $j,j'$,
$$
(\lambda_n^2(x_n^{j,+}-x_n^{j',-})_1,\lambda_n(x_n^{j,+}-x_n^{j',-})')
$$
either converges or tends to $\infty$ in norm.  In the former case, changing the $\phi^{j,-}$ if needed, we may assume that for all $j,j'$, either $x_n^{j,+} \equiv x_n^{j',-}$ or 
$$
\lambda_n^2|(x_n^{j,+}-x_n^{j',-})_1|+\lambda_n|(x_n^{j,+}-x_n^{j',-})'| \to \infty.
$$
Reordering and inserting $0$'s in place of $\phi^{j,+}$ or $\phi^{j,-}$ where needed, we may assume that $x_n^{j,+}\equiv x_n^{j,-}$ for all $j$.

We thus obtain decompositions
$$
f_n = \sum_{j=1}^J e^{ix_n^j\omega} g_n^j + r_n^J, \qquad J \in \N,
$$
as in the statement of Theorem~\ref{T:large space}.  It remains to verify (i), (ii-iii'), (iv) and (v).  

The approximation (i) of $\scriptE g_n^j$ by a rescaling and modulation of $\scriptE_\P \phi^j$ follows from \eqref{E:sphere to parab} and the triangle inequality.  The $L^2$ orthogonality, conclusion (ii-iii'), follows on each hemisphere separately by the weak limit condition; we put the pieces together via $\|f_n\|_2^2 = \|f_n^+\|_2^2+\|f_n^-\|_2^2$.  

The $L^q$-orthogonality, conclusion (v), follows by iterating the generalized Br\'ezis--Lieb lemma, Lemma~3.1 of \cite{FLS} (cf. \cite{BL}); because the lemma was developed to address precisely this situation, we will be brief in showing how that lemma applies here.  In the notation of that lemma, given $J$ and $M$, we set
\begin{gather*}
    \alpha_n^M(x):=\lambda_n^{-\frac{d+2}q}\scriptE (e^{-ix_n^J\omega}r_n^{J-1}\chi_{\{|\omega'| \leq M \lambda_n\}})(\lambda_n^{-2}x_1,\lambda_n^{-1}x')\\
    \pi_n^M(x):=\sum_\pm e^{\pm i \lambda_n^{-1}x_1}\scriptE_{\P}(\phi^{j,\pm}\chi_{|\xi| \leq M\}})(\mp x_1,x')\\
    \rho_n^M(x):=\scriptE(\lambda_n^{-\frac{d+2}q}e^{-ix_n^J\omega}r_n^J \chi_{\{|\omega'| \leq M \lambda_n\}})(\lambda_n^{-2}x_1,\lambda_n^{-1}x')\\
    \sigma_n^M(x):= \lambda_n^{-\frac{d+2}q}\scriptE (g_n^J\chi_{\{|\omega'| \leq M \lambda_n\}})(\lambda_n^{-2}x_1,\lambda_n^{-1}x')-\pi_n^M(x),
\end{gather*}
and let $\alpha_n,\pi_n,\rho_n,\sigma_n$ denote the corresponding functions with no truncation in the frequency variables.
Then $\alpha_n^M = \pi_n^M + \rho_n^M + \sigma_n^M$. We immediately see that $|\pi_n^M|$ is bounded by a fixed $L^q$ function.  Observing that
\begin{align*}
\rho_n^M(x) &= \sum_\pm e^{\pm i\lambda_n^{-2}x_1}\int e^{i(x_1,x')(\lambda_n^{-2}(\pm\sqrt{1-|\lambda_n\xi|^2}\mp 1),\xi)} e^{-ix_n^J(\lambda_n^{-2}(\pm\sqrt{1-|\lambda_n\xi|^2}\mp 1),\xi)}\\
&\qquad \times r_n^J(\pm\sqrt{1-|\lambda_n\xi|^2},\lambda_n\xi)\chi_{\{|\xi| \leq M\}} \, \tfrac{d\xi}{\sqrt{1-|\lambda_n\xi|^2}},
\end{align*}
we see that $\rho_n^J \to 0$ pointwise.  Finally, 
\begin{align*}
\sigma_n^M &= \sum_\pm e^{\pm i\lambda_n^{-2}x_1}\int(e^{i(x_1,x')(\lambda_n^{-2}(\pm\sqrt{1-|\lambda_n\xi|^2}\mp 1),\xi)} \tfrac1{\sqrt{1-|\lambda_n\xi|^2}} - e^{i(x_1,x')(\mp\tfrac12|\xi|^2,\xi)})\\
&\qquad\times\phi^{J,\pm}(\xi)\, \chi_{\{|\xi| \leq M\}}\, d\xi.
\end{align*}
If $\phi^{J,\pm}(\xi)\, \chi_{\{|\xi| \leq M\}}$ are assumed to be smooth, $\sigma_n^M \to 0$ in $L^q$ by stationary phase and the dominated convergence theorem; for general $\phi^{J,\pm}$, convergence to 0 follows from boundedness of $\scriptE$ and $\scriptE_\P$ from $L^p$ to $L^q$ and density arguments.  

By the generalized Br\'ezis--Lieb lemma, 
$$
\lim_{n \to \infty}\|\alpha_n^M\|_q^q - \|\pi_n^M+\sigma_n^M\|_q^q - \|\rho_n^M\|_q^q =0.
$$
By hypothesis \eqref{E:Efn>M large space} and $L^p \to L^q$ boundedness of $\scriptE$ and $\scriptE_\P$, 
$$
\lim_{M \to \infty} \limsup_{n \to \infty} \|\alpha_n^M-\alpha_n\|_q + \|\pi_n^M-\pi_n\|_q + \|\rho_n^M-\rho_n\|_q+\|\sigma_n^M-\sigma_n\|_q = 0.
$$
Therefore
$$
\lim_{n \to \infty}\|\alpha_n\|_q^q - \|\pi_n + \sigma_n\|_q^q - \|\rho_n\|_q^q =0,
$$
i.e., 
$$
\lim_{n \to \infty} \|\scriptE r_n^{J-1}\|_q^q - \|\scriptE g_n^J\|_q^q - \|\scriptE r_n^J\|_q^q = 0.
$$
The $L^q$ orthogonality, (iv) follows by induction.

Finally, smallness of the errors follows from (ii-iii') and Lemma~\ref{L:inv rest large scale}:  a non-negligible remainder term yields a weak limit with $L^2$ norm bounded below, and that reduces the $L^2$ norm of the subsequent remainder by a nonnegligible amount.  
\end{proof}

Now we turn to the analogue of Lemma~\ref{L:Lp almost orthog 1} for the case of antipodal frequency concentration.  Let $\psi,\rho \in C^\infty_c(\R^d;[0,1])$ with $\psi(0) = 1$ and $\int \rho = 1$.  For $r>0$, we define $\psi_r(\xi):=\psi(r\xi)$ and $\rho^r(\xi) = r^{-d}\rho(\tfrac\xi r)$.  Given a doubly indexed sequence $\{x_n^j\}_{j,n \in \N}$, and $\lambda_n \searrow 0$, we define operators on integrable functions $f$ on $\S^d$ by
$$
(\pi_r)^{j,\pm}_n f(\xi) := \rho^r *_\eta (\psi_r(\eta) e^{-ix_n^j(\pm\sqrt{1-|\lambda_n\eta|^2},\lambda_n\eta)}\lambda_n^{d/p} f(\pm\sqrt{1-|\lambda_n\eta|^2},\lambda_n\eta).
$$
We also define vector-valued  operators 
$$
(\Pi_r)_n^J f := (((\pi_r)_n^{j,\bullet} f)_{j=1}^J)_{\bullet \in \{\pm\}}, \qquad J \in \N.  
$$
We recall the notation $\tilde p := \max\{p,p'\}$. Thus $\tilde p \geq p$.

\begin{lemma} \label{L:Lp almost orthog large}
Let $1 < p < \infty$.  Assume that the sequences $\{y_n^j\}$ obey $\lim_{n \to \infty} |y_n^j-y_n^{j'}| = \infty$ for all $j \neq j'$, where $y_n^j := (\lambda_n^2(x_n^j)_1,\lambda_n (x_n^j)')$.  Then the $(\Pi_r)_n^J$ map $L^p(\S^d)$ boundedly into $\ell^p_\bullet(\ell^{\tilde p}_j(L^p(\R^d)))$, with operator norms bounded uniformly in $r,n$.  Moreover, 
\begin{equation} \label{E:Pinj 1 large space}
\lim_{r \to 0}\limsup_{n \to\infty} \|(\Pi_r)_n^J\|_{L^p \to \ell^p_\bullet(\ell^{\tilde p }_j(L^p(\R^d)))} = 1.
\end{equation}
Finally, let $\{f_n\}$ be a bounded sequence in $L^p(\S^d)$, supported in $\{\omega \in \S^d:|\omega'|<\tfrac12\}$, for which the weak limits $\phi^{j,\pm}$ in \eqref{E:phijpm} exist, and define $g_n^j$ as in \eqref{E:gnj def}.  Then 
\begin{gather} \label{E:pinj f is phi large space}
    \lim_{r \to 0} \lim_{n \to \infty} \|(\pi_r)_n^{j,\pm} f_n-\phi^{j,\pm}\|_p = 0\\ \label{E:pinj* phi is g large space} 
    \lim_{r \to 0} \lim_{n \to \infty} \|[(\pi_r)_n^{j,\pm}]^* \phi^{j,\pm} - e^{ix_n^j\omega}g_n^{j,\pm}\|_p = 0.
\end{gather}
\end{lemma}

\begin{proof}
Since $f \mapsto (f\chi_{\{\omega_1>0\}},f\chi_{\{\omega_1<0\}})$ maps $L^p(\S^d)$ boundedly into $\ell^p(L^p(\S^d) \times L^p(\S^d))$, with operator norm 1, \eqref{E:Pinj 1 large space} would follow from 
\begin{equation} \label{E:Pinj+ 1 large space}
\lim_{r \to 0}\limsup_{n \to\infty} \|(\Pi_r)_n^{J,\bullet}\|_{L^p \to \ell^{\tilde p}_j(L^p(\R^d))} = 1, \quad \bullet = +,-,
\end{equation}
and using reflection across the hyperplane $\{0\} \times \R^d$, we may choose the positive sign in \eqref{E:Pinj+ 1 large space}, \eqref{E:pinj f is phi large space}, and \eqref{E:pinj* phi is g large space}.    To keep equations within lines, we will omit the superscript $+$ from the operators from the remainder of the proof.  

It is elementary to show that $\|(\pi_r)_n^{j} f\|_p \lesssim \|f\|_p$, with implicit constant independent of $f,r,n,j$. Moreover, for any $j$,
$$\lim_{r \to 0}\limsup_{n \to \infty}\|(\pi_r)_n^{j}\|_{L^p \to L^p} = 1.
$$
Indeed, the upper bound uses the compact support of $\psi$, and the lower bound is obtained by considering a shrinking profile $e^{ix_n^j\omega} g_n^{j,+}$.  In particular, \eqref{E:Pinj+ 1 large space} holds when $p=1,\infty$.  

The upper bound in \eqref{E:Pinj 1 large space} will thus follow from that in the case $p=2$ by complex interpolation.  

We turn now to the proof that 
\begin{equation} \label{E:PinJ norm to 1 L2}
\limsup_{n \to \infty} \|(\Pi_r)_n^{J}\|_{L^2 \to \ell^2(\ell^2(L^2))} = 1,
\end{equation}
for all $r$.  We bound
\begin{equation} \label{E:PinjPink* phi}
\begin{aligned}
    &\|(\Pi_r)_n^{J}[(\Pi_r)_n^{J}]^*(\phi^j)_{j=1}^J\|_{\ell^2(L^2)}^2 \\
    &\qquad \leq (1+\eps)\sum_{j=1}^J \|(\pi_r)_n^{j}[(\pi_r)_n^{j}]^* \phi^j\|_{L^2}^2 + C_{J,\eps} \sum_{j \neq k} \|(\pi_r)_n^{j}[(\pi_r)_n^{k}]^* \phi^k\|_{L^2}^2.
\end{aligned}
\end{equation}
We expand (for $\xi \in \R^d$)
$$
(\pi_r)_n^{j}[(\pi_r)_n^{k}]^* \phi(\xi) = \int \phi(\zeta)(K_r)_n^{jk}(\zeta,\xi)\,d\xi,
$$
where 
$$
(K_r)_n^{jk}(\zeta,\xi) := \int \psi_r(\eta)^2 \rho_r(\xi-\eta)\rho_r(\zeta-\eta)e^{i(x_n^k-x_n^j)(\sqrt{1-|\lambda_n\eta|^2},\lambda_n\eta)}\sqrt{1-|\lambda_n\eta|^2}\, d\eta.
$$
A straightforward computation (using our hypotheses on $\rho,\psi$) gives 
$$\|(K_r)_n^{jk}\|_{L^\infty_\zeta L^1_\xi}, \|(K_r)_n^{jk}\|_{L^\infty_\xi L^1_\zeta} \leq 1.
$$
For $j \neq k$, stationary phase gives $\|(K_r)_n^{jk}\|_{L^\infty} \lesssim \jp{y_n^k-y_n^j}^{-\frac d2}$.  Since the $(K_r)_n^{jk}$ have supports contained in a fixed compact set (for fixed $r$),
$$\|(K_r)_n^{jk}\|_{L^\infty_\zeta L^1_\xi}, \|(K_r)_n^{jk}\|_{L^\infty_\xi L^1_\zeta} \to 0 ,
$$
as $n \to \infty$.  Hence by \eqref{E:PinjPink* phi}, 
$$
\limsup_{n \to \infty} \|(\Pi_r)_n^{j}[(\Pi_r)_n^{k}]^*\|_{\ell^2(L^2) \to \ell^2(L^2)} \leq 1+\eps,
$$
for all $\eps>0$.  Sending $\eps \to 0$, we obtain \eqref{E:PinJ norm to 1 L2}.

It remains to prove \eqref{E:pinj f is phi large space} and \eqref{E:pinj* phi is g large space}. By boundedness of the $(\pi_r)_n^k$ on $L^p$, we may assume the $\phi^k$ are compactly supported. Modulating $f_n\rightsquigarrow e^{-x_n^j\omega}f_n$ if needed, it suffices to consider the case $x_n^j=0$, for all $n$.    

Since \eqref{E:pinj* phi is g large space} is essentially elementary, we turn to \eqref{E:pinj f is phi large space}.  Noting that $|y_n^k| \to \infty$ when $k \neq j$,
$$
\lambda_n^{\frac dp} e^{ix_n^k(\sqrt{1-|\lambda_n\xi|^2},\lambda_n \xi)} g_n^{k,+}(\sqrt{1-|\lambda_n\xi|^2},\lambda_n\xi) = e^{i(x_n^k)_1}e^{iy_n^k(\lambda_n^{-2}\sqrt{1-|\lambda_n\xi|^2},\xi)}\phi^k(\xi) \rightharpoonup 0,
$$
weakly when $k \neq j$. Hence, by construction, 
$$
\lambda_n^{\frac dp} f(\sqrt{1-|\lambda_n\xi|^2},\lambda_n\xi) \rightharpoonup \phi^j,
$$
weakly, as $n \to \infty$.  Therefore $(\pi_r)_n^j f_n(\xi) \to \int \rho_r(\xi-\eta)\psi_r(\eta)\phi^j(\eta)\, d\eta$ pointwise in $\xi$, as $n \to \infty$.  In fact, the convergence is in $L^p$ by the dominated convergence theorem, and  \eqref{E:pinj f is phi large space} follows. 
\end{proof}

We are finally ready to conclude the proof of the $L^p$ profile decomposition in the case of antipodal concentration, Theorem~\ref{T:large space}.

\begin{proof}[Proof of Theorem~\ref{T:large space}]

By Lemma~\ref{L:L2 profile large scale}, we may assume that $p_1:=p \neq 2$.  We fix an exponent $p_0 < \tfrac{2(d+1)}d$ such that $p_1$ lies between $p_0$ and $2=:p_2$ and the extension conjecture holds at $(p_0,q_0)$, with $q_0:=\tfrac{d+2}d p_0'$.  

It suffices to prove the theorem under the additional assumptions that $\|f_n\|_p \leq 1$, $\lim\|f_n\|_p = 1$, and $f_n$ is supported in $\{|\omega'|<\frac12\}$, for all $n$.  (Indeed, we may simply add $f_n\chi_{\{|\omega'|>\frac12\}}$ to the error terms $r_n^J$ and then multiply by a constant.)  We define
$$
f_n^M:=f_n \chi_{\{|f_n| \leq M\lambda_n^{-d/p}\}}\chi_{\dist(\omega,[e_1])<M\lambda_n\}}.  
$$
Setting $\alpha_i:=\tfrac dp-\tfrac d{p_i}$  for $i\in\{0,1,2\}$, $\{\lambda_n^{\alpha_2}f_n\}$ obeys the hypotheses of Theorem~\ref{T:large space} with $p=2$.  Therefore, by Lemma~\ref{L:L2 profile large scale}, after passing to a subsequence (which does not depend on $M \in \N$), for all $M$, there exist $\{x_n^{j,M}\} \subseteq \R^{d+1}$ with
$$
\lim_{n \to \infty} (\lambda_n^2|(x_n^{j,M}-x_n^{j,M})_1| + \lambda_n|(x_n^{j,M}-x_n^{j,M})'|) = \infty, \qquad j \neq j',
$$
and weak limits
$$
L^2(\R^d) \ni \phi^{j,M,\pm} := \wklim \lambda_n^{d/p}e^{-ix_n^{j,M}(\pm\sqrt{1-|\lambda_n\xi|^2},\lambda_n\xi)}f_n^M(\pm\sqrt{1-|\lambda_n\xi|^2},\lambda_n\xi),
$$
such that, with
\begin{gather*}
    g_n^{j,M}(\omega):=\sum_\pm \lambda_n^{-d/p}\phi^{j,M,\pm}(\lambda_n^{-1}\omega')\chi_{\{\pm\omega_1>0\}}\chi_{\{|\omega'|<\frac12\}},\\
    r_n^{J,M}:=f_n^M-\sum_{j=1}^J e^{ix_n^{j,M}\omega}g_n^{j,M},
\end{gather*}
we have
\begin{gather}
    \notag
    \lim_{n \to \infty}\|\lambda_n^{\alpha_2}f_n^M\|_2 - \sum_{j=1}^J\sum_{\pm}\|\phi^{j,M,\pm}\|_2^2 - \|\lambda_n^{\alpha_2}r_n^{J,M}\|_2^2 = 0,
    \\
    \label{E:qi BL}
    \lim_{n \to \infty} \|\lambda_n^{\alpha_2} \scriptE f_n^M\|_{q_2}^{q_2} - \sum_{j=1}^J \|\lambda_n^{\alpha_2}\scriptE g_n^{j,M}\|_{q_2}^{q_2} - \|\lambda_n^{\alpha_2}\scriptE r_n^{J,M}\|_{q_2}^{q_2} = 0,
    \\
    \label{E:rM qi}
    \lim_{J \to \infty} \lim_{n \to \infty} \|\lambda_n^{\alpha_2}\scriptE r_n^{J,M}\|_{q_2} = 0,
\end{gather} 
for all $J$.  In fact, the generalized Br\'ezis--Lieb lemma applied in the proof of Lemma~\ref{L:L2 profile large scale} implies that \eqref{E:qi BL} holds with $\alpha_i,q_i$ in place of $\alpha_2,q_2$,  when $i=0,1$ as well.  In particular, $\{\lambda_n^{\alpha_0}\scriptE r_n^{J,M}\}$ is bounded in $L^{q_0}$, uniformly in $J$, as $n \to \infty$, and hence by H\"older's inequality, \eqref{E:rM qi} also holds with $\alpha_1, q_1$ in place of $\alpha_2,q_2$, i.e. (since $\alpha_1=0$),
\begin{equation} \label{E:rm q0}
\lim_{J \to \infty} \limsup_{n \to \infty} \|\scriptE r_n^{M,J}\|_{q_1} = 0, \qtq{for all $M$.}
\end{equation}

After passing to a subsequence (and modulating the $\phi^{j,M,\pm}$ and multiplying by unimodular constants if needed), we may assume that for any $M,j$ and $M',j'$, either $x_n^{M,j} \equiv x_n^{M',j'}$, or 
$$
\lim_{n \to \infty} \lambda_n^2|(x_n^{M,j}-x_n^{M',j'})_1| + \lambda_n|(x_n^{M,j}-x_n^{M',j'})'| = \infty,
$$
for all $j,j',M,M' \in \N$.  By reordering the profiles and inserting zero profiles as needed (using yet another diagonal argument), we may write $x_n^{j,M} = x_n^j$, for all $M$.  Passing to a further subsequence, the weak limits $\phi^{j,\pm}$ defined in \eqref{E:phijpm} all exist.  

By Lemma~\ref{L:Lp almost orthog large}, for every $\eps > 0$, and any choice of $\bullet \in \{\pm\}$, $\|\phi^{j,\bullet}\|_p \geq \eps$ for at most $\eps^{-\tilde p}$ values of $j$. Indeed,
\begin{equation} \label{E:bound phi j dot}
\sum_{j=1}^J\|\phi^{j,\bullet}\|_p^{\tilde p} = \lim_{r \to 0} \lim_{n \to \infty} \|(\Pi_r)_n^Jf_n\|_{\ell^{\tilde p}(L^p)}^{\tilde p} \leq 1.
\end{equation}
In particular, we may reorder the profiles so that $\|\phi^{j,+}\|_p^p + \|\phi^{j,-}\|_p^p$ is decreasing.  
In the notation of Theorem~\ref{T:large space}, it remains to prove (ii), (iii), (iv), and (v).  

Inequality (ii), which was one of the $L^p$ almost orthogonality conditions, follows from Lemma~\ref{L:Lp almost orthog large}.  Indeed, 
\begin{align*}
    &\sum_{\pm}\bigl(\sum_{j=1}^J\|\phi^{j,\pm}\|_p^{\tilde p}\bigr)^{p/\tilde p}
    =\lim_{r \to \infty}\lim_{n \to \infty} \|(\Pi_r)_n^J f_n\|_{\ell^p(\ell^{\tilde p}(L^p)} \leq 1.
\end{align*}
Inequality (iii), the other $L^p$ almost orthogonality estimate, follows from Lemma~\ref{L:Lp almost orthog large} as well, since
\begin{align*}
    \limsup_{n \to \infty} \|\sum_{j=1}^J e^{ix_n^j\omega}g_n^j\|_p = \lim_{r \to \infty}\lim_{n \to \infty} \|[(\Pi_r)_n^J]^*((\phi^{j,\bullet})_{j=1}^J)_{\bullet \in\{\pm\}}\|_p \leq \|((\phi^{j,\bullet})_{j=1}^J)_{\bullet \in\{\pm\}}\|_{\ell^p(\ell^{\tilde p'}(L^p)}.
\end{align*}

The $L^q$-orthogonality condition (iv) follows from the generalized Br\'ezis--Lieb argument given in the proof of Lemma~\ref{L:L2 profile large scale} (indeed, the condition $p=2$ played no role in that part of the argument).   

We are left to establish (v), smallness of the extensions of the remainder terms.  We wish to remove the dependence on $M$ in \eqref{E:rm q0}, and, as in the previous section, we begin by showing that the rate of convergence to 0 in \eqref{E:rm q0} as $J \to \infty$ is independent of $M$.  

Given $M \in \N$ and $(j,\bullet) \in \N \times \{\pm\}$, we define $\phi_M^{j,\bullet}:=\phi^{j,\bullet} \chi_{\{|\xi|<M\}}$ and 
$$
(g_M)_n^{j,\bullet}(\omega):=\lambda_n^{-d/p}\phi^{j,\bullet}(\lambda_n^{-1}\omega')\chi_{\{\bullet\omega_1>0\}}\chi_{\{|\omega'|<\tfrac12\}}.  
$$

\begin{lemma} \label{L:EphiM to Ephi}
For every $M,j \in \N$ and $\bullet \in \{\pm\}$, 
\begin{equation} \label{E:EphiM to Ephi}
\lim_{n \to \infty}\|\scriptE ((g_M)_n^{j,\bullet} - g_n^{M,j,\bullet})\|_q \leq \limsup_{n \to \infty} \sup_{E \subseteq E_n^{M}} \|\scriptE f_n \chi_E\|_q,
\end{equation}
where $E_n^M$ was defined in \eqref{E:Efn>M large space}.  
\end{lemma}

\begin{proof}[Proof of Lemma~\ref{L:EphiM to Ephi}]
We give the proof when $\bullet = +$; the other case follows by taking conjugates.  To simplify notation, we omit the $j$ and $+$ from our superscripts. 

By Lemma~\ref{L:sphere to parab}, 
$$
\lim_{n \to \infty}\|\scriptE ((g_M)_n - g_n^{M})\|_q = \|\scriptE_\P(\phi_M-\phi^{M})\|_q.  
$$

Since 
$$
\phi_M(\xi) = \wklim_{n \to \infty} \lambda_n^{d/p}f_n(\sqrt{1-|\lambda_n \xi|^2},\lambda_n\xi) \chi_{\{|\xi|<M\}},
$$
we have
\begin{align}\notag
    &\scriptE_\P(\phi_M-\phi^M)(-x_1,x') \\\notag
    &\qquad
    =\lim_{n \to \infty} \int e^{ix(-\frac12|\xi|^2,\xi)}\lambda_n^{d/p}f_n(\sqrt{1-|\lambda_n\xi|^2},\lambda_n\xi)\, \chi_{\{|\xi|<M\}}\chi_{\{|f_n|>M\lambda_n^{-d/p}\}}\, d\xi\\
    \notag
    &\qquad 
    =
    \lim_{n \to \infty} e^{-i\lambda_n^{-2}x_1}\int e^{i(\lambda_n^{-2}x_1,\lambda_n^{-1}x')(1-\frac12|\xi|^2,\xi)}\lambda_n^{-d/p'}f_n(\sqrt{1-|\xi|^2},\xi)\\
    \label{E:rhs Ephi_M-Ephi^M}
    &\qquad\qquad\qquad\qquad\qquad\times \chi_{\{|\xi|<M\lambda_n\}}\chi_{\{|f_n|>M\lambda_n^{-d/p}\}}\, d\xi.
\end{align}
By H\"older's inequality,
$$
\|\lambda_n^{-d/p'}f_n(\sqrt{1-|\xi|^2},\xi)\chi_{\{|\xi|<M\lambda_n\}}\|_{L^1_\xi} \leq M^{-d/p'}.
$$
For $|\xi|<M\lambda_n$, 
$$
|1-\sqrt{1-|\xi|^2}| \lesssim M^2\lambda_n^2, \qquad |(1-\tfrac12|\xi|^2)-\sqrt{1-|\xi|^2}| \lesssim M^4 \lambda_n^4.
$$
Thus by H\"older's inequality, \eqref{E:rhs Ephi_M-Ephi^M}, and $\frac{d}{p'} = \frac{d+2}{q}$,
\begin{align*}
&\scriptE_\P(\phi_M-\phi^M)(-x_1,x')
\\
&\:=
\lim_{n \to \infty} \lambda_n^{-(d+2)/q}e^{i\lambda_n^{-2}x_1}\scriptE(f_n\chi_{\{|\omega'|<M\lambda_n,|f_n|>M\lambda_n^{-d/p},\omega_1>0\}})(\lambda_n^{-2}x_1,\lambda_n^{-1}x'),
\end{align*}
for all $x$.  Finally, by Fatou and a change of variables,
$$
\|\scriptE_\P(\phi_M-\phi^M)\|_q \leq \|\scriptE(f_n \chi_{\{|\omega'|<M\lambda_n,|f_n|>M\lambda_n^{-d/p},\omega_1>0\}})\|_q.
$$  
\end{proof}

Let $\eps > 0$.  By \eqref{E:bound phi j dot}, $\lim_{j \to \infty} \|\phi^{j,\bullet}\|_p = 0$, so by boundedness of $\scriptE$, $\lim_{j \to \infty} \lim_{n \to \infty} \|\scriptE g_n^j\|_q = 0$.  Hence, by \eqref{E:Efn>M large space} and \eqref{E:EphiM to Ephi}, there exist $M_\eps,J_\eps$ such that $\lim_{n \to \infty} \|\scriptE g_n^{M,j}\|_q < \eps$ for all $M\geq M_\eps$ and $j > J_\eps$.  Let $M\geq M_\eps$.  By \eqref{E:rm q0}, there exists $J_{M,\eps}$ such that $\lim_{n \to \infty} \|\scriptE r_n^{M,J}\|_q < \eps$, for all $J\geq J_{M,\eps}$.  

By the definition of the remainder terms, the generalized Br\'ezis--Lieb lemma, boundedness of $\scriptE$, and the $L^p$-almost orthogonality condition (ii),
\begin{align*}
&\lim_{n \to \infty} \|\scriptE(r_n^{M,J_\eps}-r_n^{M,J_{M,\eps}})\|_q^q = 
\lim_{n \to \infty} \|\scriptE(\sum_{j=J_\eps+1}^{J_{M,\eps}} e^{ix_n^j \omega}g_n^{M,j})\|_q^q \\
& \qquad = 
\sum_{j=J_\eps+1}^{J_{M,\eps}} \lim_{n \to \infty}\|\scriptE g_n^{M,j}\|_q^q
\leq S_{p \to q}^{q-\tilde p} \eps^{q-\tilde p} \sum_{j=J_\eps+1}^{J_{M,\eps}} (\sum_{\bullet \in \{\pm\}}\|\phi^{M,j,\bullet}\|_p^p)^{\tilde p/p}
\lesssim \eps^{q-\tilde p}.
\end{align*}
Since $q>\tilde p$, after changing the value of $\eps$, we may assume that 
\begin{equation} \label{E:ErnMj unif small}
\|\scriptE r_n^{M,J}\|_q < \eps,
\end{equation}
for all $M>M_\eps$ and $J>J_\eps$.  

Finally, we transfer \eqref{E:ErnMj unif small} to the $\scriptE r_n^J$.  The $L^q$-orthogonality (iv) applied with some  $r_n^{J_0}$ in place of $f_n$ implies that $\lim_{n \to \infty} \|\scriptE r_n^J\|_q$ is non-increasing in $J$.  Thus it suffices to bound $\lim_{n \to \infty}\|\scriptE r_n^{J_\eps}\|_q$.  

By the triangle inequality, the definition of the remainder terms, and \eqref{E:ErnMj unif small}; then \eqref{E:Efn>M large space}, $\phi_M^{j,\bullet} \to \phi^{j,\bullet}$ in $L^p$, and \eqref{E:EphiM to Ephi},
\begin{align*}
    \lim_{n \to \infty}\|\scriptE r_n^{J_\eps}\|_q
    \leq O(\eps) + \lim_{M \to \infty} \lim_{n \to \infty} \|\scriptE(f_n-f_n^M)\|_q + \lim_{M \to \infty} \sum_{j=1}^{J_\eps} \|\scriptE(g_n^j-g_n^{j,M})\|_q
     = O(\eps).
\end{align*}
As we have confirmed conclusion (v), the proof of Theorem~\ref{T:large space} is complete.  

\end{proof}

\section{Proof of Theorem~\ref{T:off scaling}}\label{S:off scaling}

\begin{proof}[Proof of Theorem~\ref{T:off scaling}]
Under the hypotheses of Theorem~\ref{T:off scaling}, let $\{f_n\}$ be an $L^p$-normalized extremizing sequence. By Proposition~\ref{P:subcrit freq}, the conditions of Theorem~\ref{T:space1} apply after passing to a subsequence. Letting $\{\phi^j\}$ denote the profiles in the conclusion of that theorem,
\begin{align*}
S_{p \to q} &= \lim_{n \to \infty} \|\scriptE f_n\|_q 
=  \bigl(\sum_{j=1}^\infty \|\scriptE \phi^j\|_q^q\bigr)^{\frac1q} 
\leq S_{p \to q} \bigl(\sum_{j=1}^\infty \|\phi^j\|_p^q\bigr)^{\frac1q}.
\end{align*} 
By Theorem~\ref{T:space1} the profiles satisfy \eqref{E:Lp orthog 1}, which for an $L^p$-normalized sequence gives $\sum_j \|\phi^j\|_p^{\tilde p} \leq 1$. Further, we have assumed that $q>\max\{p,\tfrac{d+2}d p'\} \geq \tilde{p}$. Therefore, 
\begin{align*}
S_{p \to q}\leq S_{p \to q} \sup_j \|\phi^j\|_p^{1-\frac{\tilde p}q} \bigl(\sum_{j=1}^\infty \|\phi^j\|_p^{\tilde p}\bigr)^{\frac1q} \leq S_{p \to q} \sup_j \|\phi^j\|_p^{1-\frac{\tilde p}q} \leq S_{p \to q}.
\end{align*} 
 Equality must hold at each step, and thus, after reordering the profiles, $\|\phi^1\|_p = 1$, $\|\scriptE \phi^1\|_q = S_{p \to q}$, and $\phi^j \equiv 0$ for $j \neq 1$.  Hence by Theorem~2.11 of \cite{LiebLoss},   $e^{-ix_n^1}f_n \to \phi^1$, an extremizer, (strongly) in $L^p$.  
\end{proof}

\section{Antipodally concentrating profiles}\label{S:equal profiles}

In the next section, we will prove Theorems~\ref{T:scaling}, \ref{T:op norms big p}, and \ref{T:op norms small p}, which we recall concern extremizing sequences along the parabolic scaling line $q=\tfrac{d+2}d p'$.  In this section, we prove a preliminary lemma that addresses the interactions between the extensions of pairs of sequences of functions that concentrate antipodally, building on the connection between spherical and parabolic extension previously discussed in Lemma \ref{L:sphere to parab}. This will lead naturally into the proofs of Propositions~\ref{P:op norms} and~\ref{P:alpha beta}, which are also contained in this section.  

We recall that 
$$
\beta_{p,q}:= 2^{\frac1{r'}}\bigl(\tfrac{\Gamma(\frac{q+1}2)}{\sqrt \pi \Gamma(\frac{q+2}2)}\bigr)^{\frac1q},
$$
where $r:=\max\{p,2\}$.

 \begin{lemma}\label{L:poles gnj}
Let $\lambda_n \searrow 0$, let $\phi^+,\phi^-$ be $L^p$ functions on $\R^d$, and define $g_n$ as in \eqref{E:gnj def}.   Then 
\begin{gather} \label{E:poles Lp gn}
\lim_{n \to \infty} \|g_n\|_p = \bigl(\|\phi^+\|_p^p + \|\phi^-\|_p^p\bigr)^{\frac1p},\\
\label{E:poles Egn leq}
\lim_{n \to \infty} \|\scriptE g_n\|_q \leq \beta_{p,q} P_{p \to q}\bigl(\|\phi^+\|_p^p + \|\phi^-\|_p^p\bigr)^{\frac1p}, 
\end{gather}
 and equality in \eqref{E:poles Egn leq} occurs if and only if either $\phi^+ \equiv \phi^- \equiv 0$, or $p \geq 2$, $|\scriptE \phi^+(-x_1,x')| \equiv |\scriptE \phi^-(x_1,x')|$, and the $\phi^\pm$ are both extremizers for $\scriptE_\P$. 
\end{lemma}

\begin{proof}
The identity \eqref{E:poles Lp gn} follows by parametrizing the upper and lower hemispheres, rescaling, and applying the dominated convergence theorem.

Now we turn to inequality \eqref{E:poles Egn leq}.  Identity \eqref{E:gnj def}, our Lemma \ref{L:sphere to parab}, and Lemma 6.1 of \cite{FLS} imply that 
\begin{align*}
\lim_{n \to \infty} \|\scriptE g_n\|_q 
&= \lim_{\alpha \to \infty} \bigl(\|e^{i\alpha x_1}\scriptE_\P\phi^+(-x_1,x') - \scriptE_\P \phi^-(x_1,x')\|_q^q\,\bigr)^{\frac1q}\\
&= \bigl(\tfrac1{2\pi} \int_0^{2\pi} \|e^{i\theta}\scriptE_\P\phi^+(-x_1,x') - \scriptE_\P \phi^-(x_1,x')\|_q^q\, d\theta\bigr)^{\frac1q}\\
&\leq 2^{\frac12} \bigl(\tfrac{\Gamma(\frac{q+1}2)}{\sqrt \pi \Gamma(\frac{q+2}2)}\bigr)^{\frac1q} P_{p \to q} \bigl(\|\phi^+\|_p^2 + \|\phi^-\|_p^2\bigr)^{\frac12},
\end{align*}
with equality if and only if either $|\scriptE_\P\phi^+(-x_1,x')| = |\scriptE_\P\phi^-(x_1,x')|$ a.e.\ and $\phi^\pm$ are both  extremizers for \eqref{E:rest/extn parab}, or $\phi^\pm \equiv 0$.

If $p \geq 2$,
$$
\bigl(\|\phi^+\|_p^2 + \|\phi^-\|_p^2\bigr)^{\frac12} \leq 2^{\frac12-\frac1p}\bigl(\|\phi^+\|_p^p+\|\phi^-\|_p^p\bigr)^{\frac1p},
$$
with equality if and only if $\|\phi^+\|_p = \|\phi^-\|_p$.  If $p < 2$,
$$
\bigl(\|\phi^+\|_p^2 + \|\phi^-\|_p^2\bigr)^{\frac12} \leq \bigl(\|\phi^+\|_p^p + \|\phi^-\|_p^p\bigr)^{\frac1p},
$$
with equality if and only if $\phi^+ \equiv 0$ or $\phi^- \equiv 0$.  This proves \eqref{E:poles Egn leq}, and the remark following on cases of equality.
\end{proof}

Adapting the construction of the $g_n$, yields the lower bound on the extension operator norm from Proposition~\ref{P:op norms}. 

\begin{proof}[Proof of Proposition~\ref{P:op norms}]
{ If $ P_{p \to q}$ is infinite, then $S_{p \to q}$ is as well, so in this case there is nothing to prove. } 
 Let  $1 \leq p < \tfrac{2(d+1)}d$ and set $q:=\tfrac{d+2}d p'$ and suppose $P_{p \to q}<\infty$. Further
 take $t\in[0,1]$, and suppose that $\phi^{\pm}$ are chosen such that each is an extremizer for $\scriptE_\P$, $\|\phi^+\|_p=t\|\phi^-\|_p$ and $|\scriptE_\P \phi^+(-x_1,x')| \equiv t |\scriptE_\P \phi^-(x_1,x')|$. We may always construct such a pair. Indeed, an extremizer for the parabolic extension problem, $\phi^+$, exists by  \cite{BSparab}, and after setting   $\phi^-(x)= t\overline{\phi^+(-x)}$ a direct computation shows that these conditions are satisfied.   Let $\lambda_n\to 0$. Analogously to \eqref{E:gnj def} set
\begin{equation}\label{E:cand conc ext} 
    g_n^{\pm}(\omega):=   \lambda_n^{-d/p}\phi^{\pm}( \lambda_n^{-1}\omega')\chi_{\{\pm\omega_1>0\}}\chi_{\{|\omega'|<\tfrac12\}}, \qquad g_n:= g_n^{+}+ g_n^{-}. 
\end{equation}
An argument similar to \eqref{E:poles Lp gn}, shows that, as $\|\phi^+\|_p=t\|\phi^-\|_p$ and  $\phi^{+}$ is extremizing, 
\begin{align*} \label{E:cand Lp}
\lim_{n \to \infty} \|g_n\|_p = & \|\phi^+\|_p \bigl(1 + t^p\bigr)^{\frac1p}\\
=& P_{p \to q}^{-1} \|\scriptE_\P \phi^+\|_q \bigl(1 + t^p\bigr)^{\frac1p}.
\end{align*}

Next we compute 
\begin{align*}
\lim_{n \to \infty} \|\scriptE g_n\|_q 
&= \lim_{\alpha \to \infty} \bigl(\|\scriptE_\P\phi^+(-x_1,x') - t\, e^{i\alpha x_1}\scriptE_\P \phi^-(x_1,x')\|_q^q\,\bigr)^{\frac1q}\\
&= \bigl(\tfrac1{2\pi} \int_0^{2\pi} \|\scriptE_\P\phi^+(-x_1,x') -  t\, e^{i\theta} \scriptE_\P \phi^-(x_1,x')\|_q^q\, d\theta\bigr)^{\frac1q}\\
&= \|\scriptE_\P\phi^+\|_q\bigl(\tfrac1{2\pi} \int_0^{2\pi} |1 + t\, e^{i\theta}|^q \, d\theta\bigr)^{\frac1q},
\end{align*}
using that  $|\scriptE \phi^+(-x_1,x')| \equiv t |\scriptE \phi^-(x_1,x')|$. 

Thus, 
\[ \lim_{n \to \infty} \frac{\|\scriptE g_n\|_q}{\|g_n\|_p}  = P_{p \to q} \frac{ \bigl(\tfrac1{2\pi} \int_0^{2\pi} |1 + t\, e^{i\theta}|^q \, d\theta\bigr)^{\frac1q}}{ \bigl(1 + t^p\bigr)^{\frac1p}}.\]
Whence the maximum of this quantity for $t\in[0,1]$ is a lower bound for $S_{p\to q}$, the operator norm of $\scriptE$. Note that $[0,1]$ is the natural domain for $t$ as it represents the ratio of the smaller $L^p$-norm to the larger, for two concentrating profiles. 
\end{proof}

Proposition \ref{P:alpha beta} is an immediate corollary of Proposition~\ref{P:op norms} and Lemma~\ref{L:poles gnj}, which characterizes concentrating extremizing sequences when $p\geq 2$.

\subsection*{An aside on $\alpha$'s and $\beta$'s}
Let $\bar{S}_{p \to q}$ denote the supremum over all antipodally concentrating sequences $g_n$ (of the form \eqref{E:gnj def}) of the quantity 
$$
\lim_{n \to \infty} \|\scriptE g_n\|_q/\|g_n\|_p.  
$$
In this section, we have shown that $\alpha_{p,q}P_{p \to q} \leq \bar S_{p \to q} \leq \beta_{p,q}P_{p \to q}$, for $p$ and $q$ along the scaling line, that both inequalities are equalities when $p \geq 2$, and that the second inequality is strict when $1 < p < 2$.  In the latter range, however, a bit more information might lead to a sharper version of Theorem~\ref{T:op norms small p}.  Namely, two questions that seem interesting are whether $\bar S_{p \to q}$ might equal $\alpha_{p,q}P_{p \to q}$, and which value of $t$ maximizes the right hand side of \eqref{E:def alpha}.  These questions may be modified slightly into a more general framework, and we ask what is the value of the quantity
$$
\sup_{0 \not\equiv F,G \in L^q} 
\frac{\bigl(\tfrac1{2\pi} \int_0^{2\pi} \|e^{i\theta}F+G\|_q^q\,d\theta\bigr)^{1/q}}{\bigl(\|F\|_q^p+\|G\|_q^p\bigr)^{1/p}},
$$
do maximizers exist, and, if so, what are their properties, in the case $1<p<2<q$?  

Numerical computations (unpublished) due to Arthur DressenWall, an undergraduate student at Macalester, lead us to ask whether the right hand side of \eqref{E:def alpha} might be attained when $t$ is either 0 or 1, in which case we would have 
$$
\alpha_{p,q} = \max\left\{ 2^{\frac1{p'}}\bigl(\frac{\Gamma(\frac{q+1}2)}{\sqrt \pi \Gamma(\frac{q+2}2)}\bigr)^{\frac1q}, 1\right\},
$$ 
but we are not quite bold enough to formulate this as a conjecture.

\section{Proof of Theorem~\ref{T:scaling}} \label{S:scaling line}

The proof of Theorem~\ref{T:scaling} follows a similar outline to the proof of Theorem~\ref{T:off scaling}, with the added complication of handling the profiles from the case of concentration.  

\begin{proof}[Proof of Theorem~\ref{T:scaling}]
Under the hypotheses of Theorem~\ref{T:scaling}, 
we let $\{f_n\}$ be an $L^p$-normalized extremizing sequence of \eqref{E:restriction/extension}. Applying one stage of the frequency decomposition in Theorem~\ref{T:freq} (\textit{\`a la} \cite{Lieb83}),
\begin{equation} \label{E:EF_n^1+ER_n^1}
\begin{aligned}
    (S_{p \to q})^q &= \lim_{n \to \infty} \|\scriptE f_n\|_q^q = \lim_{n \to \infty} \|\scriptE F_n^1\|_q^q + \|\scriptE R_n^1\|_q^q \\
    &\leq \lim_{n \to \infty} \max\bigl\{\|\scriptE F_n^1\|_q^{q-p},\|\scriptE R_n^1\|_q^{q-p}\}(\|\scriptE F_n^1\|_q^p + \|\scriptE R_n^1\|_q^p) 
    \\
    &\leq \lim_{n \to \infty}(S_{p \to q})^q (\|F_n^1\|_p^p + \|R_n^1\|_p^p) = (S_{p \to q})^q.
\end{aligned}
\end{equation}
Passing to a subsequence, we may assume that all of the norms in \eqref{E:EF_n^1+ER_n^1} converge.  By reordering, we may assume that $\lim_{n \to \infty} \|\scriptE F_n^1\|_q \neq 0$.  As all inequalities in \eqref{E:EF_n^1+ER_n^1} must be equalities, $\scriptE R_n^1 \to 0$ in $L^q$ (first inequality), and $F_n^1$ is extremizing and $R_n^1 \to 0$ in $L^p$ (second inequality).  In other words, $f_n$ obeys the hypotheses of either Theorem~\ref{T:space1} or, after applying a sequence of rotations, of Theorem~\ref{T:large space}.  

If we are in the case of nonconcentration, described in the hypotheses of Theorem~\ref{T:space1}, we may follow the proof of Theorem~\ref{T:off scaling} from Section~\ref{S:off scaling} to see that $f_n$ converges in $L^p$ to an extremizer.

Thus, it remains to consider the case of antipodal concentration, in which, by neglecting the role of rotations, we may apply Theorem~\ref{T:large space}.  In the notation of that theorem, we have
\begin{equation} \label{E: g extrem}
\begin{aligned}
    (S_{p \to q})^q 
    &= \sum_{j=1}^\infty  \limsup_{n \to \infty} \|\scriptE g_n^j\|_q^q 
    \leq (S_{p \to q})^q \sum_{j=1}^\infty \limsup_{n\to\infty} \|g_n^j\|_p^q \\
    &= (S_{p \to q})^q\sum_{j=1}^\infty (\|\phi^{j,+}\|_p^p+\|\phi^{j,-}\|_p^p)^{q/p}
    \\&\leq (S_{p \to q})^q\sup_j (\|\phi^{j,+}\|_p^p+\|\phi^{j,-}\|_p^p)^{\frac{q-\tilde p}{p}}\lim_{n \to \infty}\|f_n\|_p^{\tilde p}\\
    &= (S_{p \to q})^q\sup_j (\|\phi^{j,+}\|_p^p+\|\phi^{j,-}\|_p^p)^{\frac{q-\tilde p}{p}} \leq (S_{p \to q})^q\cdot 1,
\end{aligned}
\end{equation}
and equality holds at each step in the above argument.  In particular, after reordering,
$$
\|\phi^{1,+}\|_p^p + \|\phi^{1,-}\|_p^p = 1,
$$
and $\phi^{j,\pm} \equiv 0$, for $j \neq 1$.  

Since 
$$
\phi^{j,\pm} = \wklim \lambda_n^{d/p}e^{-ix_n^j(\pm\sqrt{1-|\lambda_n\xi|^2},\lambda_n\xi)}f_n(\pm\sqrt{1-|\lambda_n\xi|^2},\lambda_n\xi)\chi_{\{|\xi|<\frac12 \lambda_n^{-1}\}},
$$ 
weak lower semi-continuity of $L^p$ norms  gives 
 $$
 \liminf_{n\to\infty} \|f_n\chi_{\{\pm \xi_1>0\}}\|_p\geq\|\phi^{1,\pm}\|_p. 
 $$
 As the sets $\{\pm \xi_1>0\} $ are disjoint, 
  $$
  1=\liminf_{n\to\infty}  \sum_{\bullet = +,-}\|f_n\chi_{\{\bullet \xi_1>0\}}\|_p^p\geq\|\phi_1^{1,+}\|_p^p+\|\phi_1^{1,-}\|_p^p\geq 1. 
  $$
Therefore, 
$$
\|\phi^{1,\pm}\|_p = \lim_{n\to\infty} \|f_n\chi_{\{\pm \xi_1>0\}}\|_p.
$$ 
Weak convergence plus convergence of norms implies  strong convergence, i.e., 
 \[  
 \lambda_n^{d/p}e^{-ix_n^1(\pm\sqrt{1-|\lambda_n\xi|^2},\lambda_n\xi)}f_n(\pm\sqrt{1-|\lambda_n\xi|^2},\lambda_n\xi)\chi_{\{|\xi|<\frac12 \lambda_n^{-1}\}} \to \phi_j^{k,\pm}, \qtq{in $L^p$,}
 \] 
 which completes the proof of Theorem~\ref{T:scaling}.
\end{proof}

We also have a bit more information.  After inserting the equation $\lim\|g_{n}^1\|_p^q =1$, equality in  \eqref{E: g extrem} becomes 
\begin{equation}\label{E:gn eq}
    \lim_{n \to \infty} \|\scriptE g_n^1\|_q =S_{p \to q} \|g_n^1\|_p = S_{p \to q}\bigl(\|\phi^+\|_p^p + \|\phi^-\|_p^p\bigr)^{\frac1p}
\end{equation} 

Depending on the values of $p, S_{p \to q}$, and $P_{p \to q}$, this equality has different implications. 

\begin{proof}[Proof of Theorems  \ref{T:op norms big p} and \ref{T:op norms small p}]

If $S_{p \to q}> \beta_{p \to q} P_{p \to q}$, then \eqref{E:gn eq} and \eqref{E:poles Egn leq} create a contradiction, ruling out the possibility of concentration. In this case, extremizers exist, and extremizing sequences possess convergent (modulo symmetries) subsequences.  

If $p\geq 2$ and $S_{p \to q}= \alpha_{p \to q} P_{p \to q}= \beta_{p \to q} P_{p \to q}$, then \eqref{E:gn eq} implies that the equality case of  Lemma~\ref{L:poles gnj} holds, which prescribes the manner of concentration. 

If $p< 2$ and $S_{p \to q}= \alpha_{p \to q} P_{p \to q}$, then the construction from the proof of Proposition \ref{P:op norms} gives a possible case of equality. 
\end{proof}


\end{document}